\documentclass{amsart}

\ifx\pdfoutput\undefined
  \usepackage[dvips]{graphicx}
\else
  \usepackage[pdftex]{graphicx}
\fi

\usepackage{subfigure} 

\usepackage{setspace}
\usepackage{fullpage}
\usepackage{amsmath} 
\usepackage{latexsym}
\usepackage{verbatim}
\usepackage{amsthm}
\usepackage{graphicx}
\usepackage{amsmath,amscd}
\usepackage{amsfonts}
\usepackage{amssymb}
\usepackage{mathrsfs}
\usepackage[all]{xy}
\usepackage{ifpdf}
\usepackage{ifthen}
\usepackage{marginnote}
\usepackage{calc}
\usepackage{ednotes}
\usepackage{comment}
\usepackage{xspace}
\usepackage{ifthen}
\usepackage{xkvltxp}
\usepackage[draft,footnote,index,nomargin]{fixme}
\usepackage{todo}
\usepackage{hyperref}

\numberwithin{equation}{section}
\swapnumbers

\theoremstyle{plain}

\newtheorem{theorem}[equation]{Theorem}
\newtheorem{corollary}[equation]{Corollary}

\newtheorem{lemma}[equation]{Lemma}
\newtheorem{prop}[equation]{Proposition}
\newtheorem{definition}[equation]{Definition}
\newtheorem{example}[equation]{Example}
\newtheorem{remark}[equation]{Remark}

\newtheorem{observation}[equation]{Observation}

\newtheorem{mainthm}{Theorem}

\DeclareMathOperator{\C}{\mathcal{C}}
\DeclareMathOperator{\M}{\mathcal{M}}
\DeclareMathOperator{\N}{\mathcal{N}}

\DeclareMathOperator{\D}{\mathcal{D}}
\DeclareMathOperator{\coll}{Coll}

\DeclareMathOperator{\I}{\mathcal{I}} 

\DeclareMathOperator{\Mod}{Mod}

\DeclareMathOperator{\Q}{\mathtt{Q}}
\DeclareMathOperator{\R}{\mathtt{R}}

\DeclareMathOperator{\unit}{\mathbf{1}}
\DeclareMathOperator{\End}{\mathit{End}}
\DeclareMathOperator{\Hom}{\mathit{Hom}}

\DeclareMathOperator{\Map}{Map}

\DeclareMathOperator{\Mor}{Mor}

\DeclareMathOperator{\id}{id}

\DeclareMathOperator{\Ho}{Ho}

\newcommand{\cat}{\mathcal}    
\newcommand{\cA}{{\cat A}}      
      
\newcommand{\cC}{{\cat C}}
\newcommand{\cD}{{\cat D}}

\newcommand{\cQ}{{\cat Q}}
\newcommand{\cR}{{\cat R}}
\newcommand{\cS}{{\cat S}}

\newcommand{\cP}{{\cat P}}

\newcommand{\Sets}{{\cat Set}}
\newcommand{\cCat}{{\cat Cat}}
\newcommand{\sSets}{s{\cat Set}}

\newcommand{\Ab}{{\cat Ab}}
\newcommand{\multi}{{\cat Multi}}

\DeclareMathOperator{\obj}{obj} 

\DeclareMathOperator{\Alg}{Alg}

\newcommand{\ovcat}{\downarrow}

\newcommand{\ra}{\rightarrow}                   
\newcommand{\lra}{\longrightarrow}              
\newcommand{\lla}{\longleftarrow}               
\newcommand{\llra}[1]{\stackrel{#1}{\lra}}      
\newcommand{\llla}[1]{\stackrel{#1}{\lla}}      

\newcommand{\we}{\llra{\sim}}                   
\newcommand{\bwe}{\llla{\sim}}
\newcommand{\cof}{\rightarrowtail}              
\newcommand{\fib}{\twoheadrightarrow}           
\newcommand{\bfib}{\twoheadleftarrow} 
\newcommand{\bcof}{\leftarrowtail}             
\newcommand{\trfib}{\stackrel{\sim}{\fib}}
\newcommand{\trcof}{\stackrel{\sim}{\cof}}
\newcommand{\btrfib}{\stackrel{{\sim}}{\bfib}}
\newcommand{\btrcof}{\stackrel{{\sim}}{\bcof}}

\usepackage{color}

\numberwithin{equation}{section}

\begin{document}

\title{Spaces of Operad Structures}         

\author[M.D. Robertson]{Marcy D. Robertson}
\address{University of Western Ontario\\
Department of Mathematics\\
Middlesex College\\ London, Ontario
\\Canada}
\email{mrober97@uwo.ca}
\begin{abstract}The purpose of this paper is to study the derived category of simplicial multicategories with arbitrary sets of objects (also known as, colored operads in simplicial sets). Our main result is a derived Morita theory for operads--where we describe the derived mapping spaces between two multicategories $\cP$ and $\Q$ in terms of the nerve of a certain category of $\cP$-$\Q$-bimodules. As an application, we show that the derived category possesses internal $\Hom$-objects. \end{abstract} 
\maketitle
\sloppy
Operads are combinatorial devices that encode families of algebras defined by multilinear operations and relations. Common examples are the operads $\mathbb{A}$, $\mathbb{C}$ and $Lie$ whose
algebras are associative, associative and commutative, and Lie algebras, respectively. Morphisms between operads systematically encode relations between different kinds of algebras. A well-studied example is the sequence $Lie\lra\mathbb{A}\lra\mathbb{C}$ which encodes the property that any commutative algebra is an associative algebra and that commutators in an associative algebra yield a Lie algebra. Multicategories, also known as colored operads, encode the laws of more complicated algebraic structures such as operadic modules, enriched categories, and even categories of operads themselves. In particular, multicategories provide a device for systematically studying morphisms between operads.\

Operads are also a generalization of classical rings to a homotopy theoretic setting. Multicategories are simply operads with ``many objects,'' analogous to pre-additive categories being rings with ``many objects.'' Taking this point of view, we have many tools available to study operads, including their representation theory. The main purpose of this paper is to prove a type of Morita theory for multicategories. We will be more explicit in Section~\ref{Moritatheory}, but the general idea of Morita theory is that equivalences between categories of representations ${}_{R}\Mod\lra {}_{S}\Mod$ correspond to some, geometrically meaningful, notion of Morita equivalence between rings $R\lra S$. Moreover, these Morita equivalences are completely characterized by families of $R$-$S$-bimodules and thus are easy to identify and study. \

To elaborate, recall that given an Abelian group $G$, we know that the group $\End(G)$ has a natural ring structure. It is well knows that there exists a natural bijective correspondence between $R$-module structures on $G$ and ring homomorphisms $R\lra \End(G)$. Similarly, we know that given $R$-modules $M$ and $N$, the set of Abelian group homomorphisms $Hom_{\Ab}(M,N)$ has a natural structure as an $R$-$R$-bimodule, and there exists a bijective correspondence between $R$-module homomorphisms $M\rightarrow N$ and $R$-$R$-bimodule maps $R\rightarrow Hom_{\Ab}(M,N).$ Morita theory proves that functors between ${}_{R}\Mod$ and ${}_{S}\Mod$ which commute with colimits are in bijective correspondence with $R$-$S$-bimodules. \ 

All of these facts have a direct analogue in operadic algebra, namely that given any space $X$, we can define the endomorphism operad $\End_{X}$ characterized by the natural bijective correspondence between $\cP$-algebra structures on $X$ and operad homomorphisms $\cP\rightarrow\End_{X}.$ Given two $\cP$-algebras $A$ and $B$, there exists a $\cP$-$\cP$-bimodule $\End_{A,B}(n):=Hom_{\C}(A^{\otimes n},B)$ such that $\cP$-algebra maps $A\rightarrow B$ are in bijective correspondence with $\cP$-$\cP$-bimodule maps $\cP\rightarrow \End_{A,B}.$\ 

Operads, however, are homotopy theoretic objects, and thus we want to add the phrase ``up to homotopy'' to every statement in this discussion. It follows that our main objects of study are the \emph{spaces of morphisms} between two multicategories $\cP$ to $\Q$ up to weak equivalence, i.e. the \emph{homotopy function complex} $\Map^{h}(\cP,\Q)$ (see Section~\ref{MappingSpaces}). The main result of this paper is the following. 

\begin{mainthm}The derived mapping space $\Map^h(\cP,\Q)$ is weakly homotopy equivalent to the moduli space of right quasi-free $\cP$-$\Q$-bimodules, ${}_{\cP}\M_{\Q}$.\end{mainthm} 

In particular, we can state that two objects are Morita equivalent if, and only if, they lie in the same connected component of the space $\Map^{h}(\cP,\Q)$. Theorem $A$ is therefore an operadic version of the derived Morita theory of To\"{e}n~\cite{Toen}. This is different than, but entangled with, the theory which is called derived Morita theory by Berger-Moerdijk in their paper~\cite{BM08}. More explicitly, Berger and Moerdijk provide a list of conditions that imply two operadic algebras have equivalent derived categories. These conditions make an appearance in this paper, in Lemma~\ref{morita2} and Proposition~\ref{morita1}, but in a different form, as the desired end is not to study the derived category, but rather the associated simplicial category obtained via Dwyer-Kan localization. \ 

\subsection{Related Work and Applications}The main technical tool in this paper is a cofibrantly generated model structure on the category of all small simplicial multicategories~\cite{Me1}. The weak equivalences are a blend of weak equivalences of operads (cf.\cite{Rezk, BM07}) and categorical equivalences. Many of the results in this paper can be extended to more operads which take values other monoidal model categories, as long as one makes additional some necessary technical adjustments. \


We also note that several results in this paper are generalizations of a recent preprint by Dwyer-Hess~\cite{DH}. They restrict to non-symmetric operads with some connectivity, and use a description of the mapping spaces to prove that the space of tangentially straightened long knots is equivalent to the double loop space of a moduli space of bimodules. 

Our use of symmetric actions in this paper allows us to say a little more than~\cite{DH} about the structure of the derived mapping spaces between operads, in particular, we can describe the internal hom-objects of the homotopy category of all small simplicial multicategories. In addition, we also prove a cosimplicial model for the mapping space, making precise an observation by Berger-Moerdijk in their paper~\cite[6]{BM07}. This cosimplicial model provides filtrations of the derived mapping space  which are necessary for computations in ~\cite{Me3}.\


\subsection{Notation and Conventions}Multicategories are frequently referred to as \emph{colored operads}, or simply \emph{operads} in the literature. There are times in this paper where the author uses the word operad or multicategory (interchangeably) without explicitly mentioning sets of objects. When it plays an important role, sets of objects will always be specified. At all other times, the result holds for general sets of objects. \

We will use interchangeably the notation for a category and its nerve. As such, we follow the convention that a functor will be a \emph{weak equivalence} if it induces a weak homotopy equivalence on the respective nerves. We will always use the phrase \emph{weak homotopy equivalence} to refer to a weak equivalence of simplicial sets in the standard (Kan) model structure. Given a model category $\C$ it makes sense to consider the (not full) subcategory $w\C$ of $\C$ which is the category with the same objects as $\C$ and morphisms the weak equivalences between objects in $\C$. We call $w\C$ the \emph{moduli category} of $\C$. The \emph{moduli space} of $\C$ will be the nerve of $w\C$.\

An adjoint pair $E:\M\leftrightarrows\mathcal{B}:U$ of functors between model categories is a \emph{Quillen pair} if $E$ preserves cofibrations and trivial cofibrations (equivalently, if $U$ preserves fibrations and trivial fibrations). The pair $(E,U)$ forms a Quillen equivalence if for all cofibrant $B\in\mathcal{B}$ and fibrant $M\in\M$, a map $EM\lra B$ is a weak equivalence in $\mathcal{B}$ if, and only if, the adjoint $M\lra UB$ is a weak equivalence in $\M$.\

Given that the right adjoint $U:\mathcal{B}\longrightarrow\M$ preserves all weak equivalences, we have an induced functor on the moduli categories $wU:w\mathcal{B}\rightarrow w\M$, and consequently an induced map of moduli spaces $wU:w\mathcal{B}\longrightarrow w\M$. If $U$ is part of a Quillen pair then the induced morphism $wU$ is a weak homotopy equivalence.\
    
If $F:\C\rightarrow\D$ is a functor and $X$ is an object of $\D$, $F\searrow X$ denotes the \emph{over category} of $F$ with respect to $X$. Objects of this category are pairs $(Y, g)$ where $Y\in \C$ and $g$ is a map $F(Y)\rightarrow X$ in $\D$. A morphism $(Y,g)\rightarrow(Y', g')$ is a map $Y\to Y'$ in $\C$ rendering the appropriate diagram commutative. The dual notion of \emph{under category} is denoted $X\searrow F$. If $F$ is the identity functor on $\D$, we write $\D\searrow X$, respectively, $X\searrow\D$. We take the following argument to be standard. Suppose that $F:\C\rightarrow\D$ is a functor such that for every morphism $h:X\rightarrow X'$ in $\D$ the map $F \searrow X\rightarrow F\searrow X'$ induced by composition with $h$ is a weak homotopy equivalence. Then one can apply Quillen's Theorem $B$ to show that for any $X\in \D$ the homotopy fiber of (the nerve of) $F$ over the vertex of $\D$ represented by $X$ is naturally weakly homotopy equivalent to $F\searrow X$. A similar result holds with over categories replaced by under categories (see~\cite[5.2]{GJ}). \

\section{Homotopy Function Complexes}\label{MappingSpaces}Given two objects $X$ and $Y$ in a model category $\C$, there is an associated simplicial set $\Map^h_{\C}(X,Y)$ called a \emph{homotopy function complex} or \emph{derived mapping space} from $X$ to $Y$. In~\cite{DK1,DK2,DK3} Dwyer and Kan show that $\Map^{h}(X,Y)$ can be computed as the nerve of a category of ``zig-zags,'' i.e. a category whose objects are zig-zags $[X\btrfib U\lra V \btrcof Y]$ and where the maps are natural transformations of diagrams which are the identity on $X$ and on $Y$. There are many variations of these zig-zag categories, including those with objects $[ X \bwe U \lra V \bwe Y]$ , or zig-zags $[ X \lra U \btrfib V \lra Y]$. As it happens, all of these variations will have homotopy equivalent nerves, and thus all of these variations have nerves homotopy equivalent to $\Map^{h}(X,Y)$. We will also require several, more rigid, models for $\Map^{h}(X,Y)$.\

Let $c\C$ denote the Reedy model structure on the category of cosimplicial objects in $\C$~\cite[Chapter 15]{Hir03}. For any object $X$ in $\C$ we write $cX$ for the constant cosimplicial object consisting of $X$ in every dimension with identity maps for all co-face and co-degeneracies. A \emph{cosimplicial resolution} of $X$ in $\C$ is a Reedy cofibrant replacement $QX^\bullet\rightarrow cX$ in $c\C$. Given such a cosimplicial resolution and an object $Y$ in $\C$ we form the simplicial set $\C(Q^\bullet X,Y)$ given by $[n]\mapsto\C(Q^nX,Y)$. If $Y\rightarrow Y'$ is a weak equivalence between fibrant objects then the induced map $\C(Q^\bullet X,Y)\longrightarrow\C(Q^\bullet X,Y')$ is a weak equivalence of simplicial sets.\ 

Now, for a fixed object $X\in\C$, let $Q(X)$ denote the category whose objects are pairs $[Q,Q\rightarrow X]$ where $Q$ is some cofibrant object in $\C$ and $Q\rightarrow X$ is a weak equivalence. For any object $Y$ in $\C$, we have a functor $$\C(-,Y):Q(X)^{op}\longrightarrow\Sets$$which sends $[Q,Q\rightarrow X]$ to $\C(Q,Y)$. We can regard this functor as taking values in $\sSets$ by composing with the embedding $\Sets\rightarrow\sSets$. We can now consider the simplicial set $\textrm{hocolim}_{Q(X)^{op}}\C(-,Y)$, where here we model the $\textrm{hocolim}$ functor by first taking the simplicial replacement of a diagram and then applying geometric realization. In dimension $n$ the simplicial replacement consists of diagrams of weak equivalences $Q_0\leftarrow Q_1\leftarrow...\leftarrow Q_n$ \emph{over} X, where each $Q_i\rightarrow X$ is in $Q(X)$, together with a map $Q_0\rightarrow Y$. In other words, the simplicial replacement is the same as the nerve of the category for which an object is a zig-zag $[X\bwe Q \ra Y]$, where $Q$ is cofibrant and $Q\ra X$ is a weak equivalence. The maps from $[X\bwe Q \ra Y]$ to $[X\bwe Q' \ra Y]$ are just maps $Q'\ra Q$ which fit into the usual commutative diagram. If $Y$ is fibrant, it is well known that this simplicial set is weakly equivalent to $\Map^h(X,Y)$ (see~\cite{D,DK3}). A dual argument shows that, if $X$ is cofibrant, $\Map^{h}(X,Y)$ is weakly equivalent to the nerve of the zig-zag category $[X\ra Q\bwe Y]$.\ 

So that we do not have to limit ourselves to only studying mapping spaces with cofibrant source or fibrant target we make use of the following proposition.\

\begin{prop}\label{rigidify1}~\cite[2.6]{DH}Let $\C$ be a left proper model category, and let $X$ and $Y$ be objects in $\C$ such that $X_c\coprod Y\longrightarrow X\coprod Y$ is a weak equivalence. Then $Map^h(X,Y)$ is weakly homotopy equivalent to the nerve of the zig-zag category $[X\ra Q\bwe Y]$.\end{prop}

We will only make brief use of the cosimplicial model for $\Map^{h}(X,Y)$ in this paper, but this is the more convenient model for computation. Let $Q^{\bullet} X\longrightarrow X$ be a cosimplicial resolution of $X$ in $c\C$. We want to relate the simplicial set $\C(Q^{\bullet} X,Y)$ to the zig-zags of categories considered above. For a given simplicial set $K$, let $\Delta K$ be the category of simplices of $K$, i.e. the over category $(S\ovcat K)$, where $S:\Delta \rightarrow \sSets$ is the functor $[n]\mapsto \Delta[n]$. The nerve of $\Delta K$ is naturally weakly equivalent to $K$ (see~\cite[text prior to Prop. 2.4]{D}). There is a functor sending $\Delta \cC(Q^{\bullet}X,Y)$ to another zig-zag category, which sends $([n],Q^nX\ra Y)$ to $[X\btrfib Q^nX \lra Y]$. \

\begin{prop}Let $Q^{\bullet}X\rightarrow X$ be a Reedy cofibrant resolution of $X$ and let $Y$ be a fibrant object of $\C$. Then $\Delta\C(Q^{\bullet}X,Y)$ is weakly equivalent to $\Map^{h}(X,Y)$.\end{prop}\begin{proof}The result is proven in \cite{DK3}, but see also \cite[Thm. 2.4]{D}.\end{proof}

\section{Operads and Multicategories}The basic idea of a multicategory is very like the idea of a category, it has objects and morphisms, but in a multicategory the source of a morphism can be an arbitrary finite sequence of objects rather than just a single object.\

A \emph{multicategory}, $\mathcal{P},$ consists of the following data:
\begin{itemize}
  \item a set of objects $\obj(\cP)$;
  \item for each $n\ge0$ and each sequence of objects $x_{1},...,x_{n},x$ a \emph{set} $\mathcal{P}(x_{1},...,x_{n};x)$ of operations which take $n$ inputs ($x_{1},...,x_{n}$) to a single output (the object $x$).\end{itemize}

These operations are equipped with structure maps for units and composition. Specifically, if $I=\{*\}$denotes the one-point set, then for each object $x$ there exists a unit map $\eta_{x}:I\rightarrow\cP(x;x)$ taking $*$ to $1_{x}$, where $1$ denotes the unit of the symmetric monoidal structure on the category $\Sets$. The composition operations are given by maps$$\mathcal{P}(x_{1},...,x_{n};x)\times\mathcal{P}(y_{1}^{1},...,y_{k_{1}}^{1};x_{1})\times\cdots\times\mathcal{P}(y_{1}^{n},...,y_{k_{n}}^{n};x_{n})\longrightarrow\mathcal{P}(y_{1}^{1},...,y_{k_{n}}^{n};x)$$which we denote by $$p,q_{1},...,q_{n}\mapsto p(q_{1},...,q_{n}).$$The structure maps satisfy the associativity and unitary coherence conditions of monoids. A \emph{symmetric multicategory} is a multicategory with the additional property that the operations are equivariant under the permutation of the inputs. Explicitly, for $\sigma\in\Sigma_{n}$ and each sequence of objects $x_{1},...,x_{n},x$ we have a right action of $\Sigma_{n}$, i.e., a morphism $\sigma^{*}:\mathcal{P}(x_{1},\cdots,x_{n};x)\rightarrow\mathcal{P}(x_{\sigma(1)},...,x_{\sigma(n)};x)$. The action maps are well behaved, in the sense that all composition operations are invariant under the $\Sigma_n$-actions, and $(\sigma\tau)^{*}=\tau^{*}\sigma^{*}$.\

In practice, one often uses the following, equivalent, definition of the composition operations, given by:$$\xymatrix{\mathcal{P}(c_{1},\cdots,c_{n};c)\times\mathcal{P}(d_{1},\cdots,d_{k};c_{i})\ar[r]^{{\circ_{i}\,\,\,\,\,\,\,\,\,\,\,}} & \mathcal{P}(c_{1},\cdots,c_{i-1},d_{1},\cdots,d_{k},c_{i+1},\cdots,c_{n};c).}$$ \

All of our definitions will still make sense if  we ask that the $k$-morphisms $\cP_n(x_1,\ldots,x_n;x)$ take values in a symmetric monoidal category other than sets; the examples we are interested in take values in either categories, symmetric spectra or simplicial sets. Multicategories whose operations take values in $\C$ are called \emph{multicategories enriched in $\C$} or \emph{$\C$-multicategories}. In particular, the strong monoidal functor $\Sets\longrightarrow\mathcal{C}$ that sends a set $S$ to the $S$-fold coproduct of copies of the unit of $\C$ takes every multicategory to a $\mathcal{C}$-enriched multicategory.\footnote{Note that a multicategory enriched over small categories can be considered enriched over simplicial sets by applying the nerve functor to the $n$-operations, since the nerve functor preserves categorical products.}\

A morphism between enriched, symmetric multicategories $F:\cP\longrightarrow\cQ$, or \emph{multifunctor}, consists of a \emph{set map} of objects $F_{0}:\obj(\cP)\longrightarrow\obj(\cQ)$ together with a family of $\Sigma_n$-equivariant $\C$-morphisms$$\{F:\cP(d_{1},...,d_{n};d)\longrightarrow\cQ(F(d_{1}),...,F(d_{n});F(d))\}_{d_{1},...,d_{n},d\in\cP}$$which are compatible with the composition structure maps. When $\cP$ and $\Q$ are enriched over simplicial sets, the multifunctor is enriched when the maps on $n$-operations preserve the enrichment. We denote the category of all small symmetric multicategories enriched in $\C$ by $\multi(\C)$ and denote the morphisms between two objects as $\multi(\cP,\Q)$.\ 

\subsection{Structure of $\multi(\C)$}Multicategories are often called colored operads, or just operads(See, for example,~\cite{BM06, BV, May, CGMV}, etc.), but we use the term multicategory in this paper because we want to emphasize the relationship between multicategory theory with classical category theory. Informally, we can say that inside every multicategory lies a category which makes up the linear part (i.e. the $1$-operations). We make this explicit by assigning to each multicategory $\cP$ a category $[\cP]_{1}$ with the same object set as $\cP$ and with morphisms given by $[\cP]_{1}(p,p')=\cP(p;p')$ for any two objects $p,p'$ in $\cP$ (i.e. just look at the operations of $\cP$ which have only one input). The functor $[-]_{1}$ takes all higher operations, i.e. $\cP(p_1,...,p_n;p)$, to be trivial. Composition and identity operations are induced by $\cP$.\

This relationship with category theory is useful in making sense of ideas which do not have obvious meaning in the multicategory setting. For example, we will often want to discuss the ``connected components'' of a multicategory, but it is difficult to say that an $n$-ary operation $\phi$ is an ``isomorphism'' in $\cP$. This is where the relationship between categories and multicategories can be useful, we can say that $\phi$ is an \textbf{isomorphism} in $\cP$ if $[\phi]_{1}$ is an isomorphism in the category $[\cP]_{1}$.\ 

\begin{definition}\label{fullyfaithful}Let $\cP$ and $\Q$ be two multicategories. A multifunctor $F:\cP\rightarrow\Q$ is \emph{essentially surjective} if $[F]_{1}$ is essentially surjective as a functor of categories. We say that $F$ is \emph{full} if for any sequence $p_1,...,p_n,p$ the function $F:\cP(p_1,...,p_n;p)\rightarrow\Q(Fp_{1},...,Fp_{n};Fp)$ is surjective. We say that $F$ is \emph{faithful} if for any sequence $p_1,...,p_n,p$ the function $F:\cP(p_1,...,p_n;p)\rightarrow\Q(Fp_{1},...,Fp_{n};Fp)$ is injective. The multifunctor $F$ is called \emph{fully faithful} if it is both full and faithful.\end{definition}\begin{definition}Let $F:\cP\rightarrow\Q$ be a functor between two symmetric multicategories. We say that $F$ is an \emph{equivalence of multicategories} if, and only if, $F$ is both fully faithful and essentially surjective.\end{definition}


Let $\C$ be the category of simplicial sets with the standard model structure. Given a simplicial category $\cA$, we can form a genuine category $\pi_0(\cA)$ which has the same set of objects as $\cA$ and whose set of morphisms $\pi_{0}(\cA)(x,y):=[\unit, \cA(x,y)]$. This induces a functor $\pi_{0}(-):\cCat(\C)\rightarrow \cCat,$ with values in the category of small categories and, moreover, a functor $\Ho(\cCat(\C))\longrightarrow\Ho(\cCat).$ In other words, any $F:\cC\longrightarrow\cD$ in $\Ho(\cCat(\C))$ induces a morphism $\pi_{0}(\C)\longrightarrow\pi_{0}(\D)$ which is well defined up to a non-unique isomorphism. This lack of uniqueness will not be an issue, as we are only interested in properties of functors which are invariant up to isomorphism.\ 

As with the non-enriched case, we can consider the linear part of a simplicial multicategory $\cP$, $[\cP]_{1}$, which is in this case a simplicial category. Applying the functor $\pi_{0}$ to the simplicial category $[\cP]_{1}$ gives us the \emph{underlying category} of the multicategory $\cP$. In order to cut back on notation, we denote this category by $[\cP]_{1}$ rather than $\pi_{0}([\cP]_{1})$.\ 

\begin{theorem}~\cite{Me1} The category of small $\C$-enriched symmetric multicategories admits a right proper cofibrantly generated model category structure in which a multifunctor $F:\cP\lra\Q$ is a weak equivalence if: \

\begin{description}	
	\item[W1] for any $n\ge 0$ and for any signature $x_1,...,x_n;x$ in $\cP$ the map of $\C$-objects $$F:\cP(x_1,...,x_n;x)\longrightarrow\cQ(Fx_1,...,Fx_n;Fx)$$ is a weak equivalence in the model category structure on $\C$. 	
	\item[W2] the induced functor $[F]_{1}$ is a weak equivalence of categories.
\end{description} A simplicial multifunctor $F:\cP\longrightarrow\Q$ is a fibration if:\begin{description}\label{fibrations}
  	\item[F1] for any $n\ge 0$ and for any signature $x_1,...,x_n;x$ in $\cP$ the map of $\C$-objects $$F:\cP(x_1,...,x_n;x)\lra\Q(Fx_1,...,Fx_n;Fx)$$ is a fibration in the model category structure on $\C$. 
	\item[F2] the induced functor $[F]_{1}$ is a fibration of categories.
\end{description} The cofibrations (respectively, acyclic cofibrations) are the multifunctors which satisfy the left lifting property (LLP) with respect to the acyclic fibrations (respectively, fibrations).\end{theorem} The fibrant objects in $\multi(\C)$ are those objects which are locally fibrant, i.e. $\cP(x_1,...,x_n;x)$ is a Kan complex for each $n\ge0$ and each sequence of objects $x_1,...,x_n;x$ in $\cP$. 

\begin{lemma}\label{fibrant objects}There exists a fibrant replacement functor on $\multi(\C)$ which fixes objects, i.e. $(\cP)_{f}\lra\cP$ is the identity on object sets.\end{lemma}

\subsection{Cofibrant Replacements}There exists several explicit cofibrant resolutions of operads in literature, and in this paper we will focus on two, the cotriple resolution and the $W$-construction of Boardman and Vogt.\

Let $\cP$, $\Q$ and $\R$ be simplicial multicategories, let $X$ be a $\cP$-$\Q$-bimodule, and let $Y$ be a $\R$-$\cP$-bimodule. We define the \emph{bar complex} $B(X,\cP,Y)$ to be the simplicial object in the category of $\R$-$\Q$-bimodules with $n^{th}$-degree $B_n(X,\cP,Y)=X\circ\cP^{\circ n}\circ Y$ with the obvious face and degeneracy maps. Applying the diagonal, we get an $\cR$-$\Q$-bimodule together with an augmentation map $\eta:diag (B(X,\cP, Y))\lra X\circ_{\cP}Y$.\

\begin{prop}The bar complex $diag(B(X,\cP,Y))$ is cofibrant in the category of $\cR$-$\Q$-bimodules and the augmentation map $\eta:diag (B(X,\cP, Y))\longrightarrow X\circ_{\cP}Y$ is a weak equivalence.\end{prop} 

\begin{proof}See ~\cite{Rezk} and ~\cite{DH}.\end{proof}

The \emph{Hochschild resolution} of a simplicial multicategory $\cP$ is a simplicial object in the category of $\cP$-$\cP$-bimodules with $B_n(\cP,\cP,\cP):=\cP^{\circ (n+2)}$ where face maps come from the composition of $\cP$ and degeneracy maps come from the unit maps of $\cP$. To shorten notation we will denote $B_n(\cP,\cP,\cP)$ by $H_n\cP$.\

The diagonal of $H_*(\cP)$, denoted $diag(H_*(\cP))$, is a $\cP$-$\cP$-bimodule with $n$-simplicies the $n$-simplicies of $H_n(\cP)$. The composition operations of $\cP$ induce maps $\cP^{\circ n}\lra\cP$. Composition and identity maps are preserved by taking diagonals, so we have natural maps $\eta:diag(H_*\cP)\lra\cP$ and $\cP\circ\cP\lra H_n\cP$. The maps $\cP\circ\cP\lra H_n\cP$ come from the image of $H_0\cP$ under degeneracy maps. Taken together, all of the degeneracy maps induce a \emph{basepoint} $\cP\circ\cP\lra diag(H_*\cP)$. It follows from arguments similar to~\cite[17.2.2. 17.2.3]{FresseBook},\cite[5.2]{DH},\cite[5]{Rezk} that $diag(H_*\cP)$ is cofibrant as a pointed $\cP$-$\cP$-bimodule and that the augmentation map $\eta:diag(H_*\cP)\lra\cP$ is a weak equivalence of pointed $\cP$-$\cP$-bimodules. Note that being cofibrant as a pointed $\cP$-$\cP$-bimodule is equivalent to saying that $\cP\circ\cP\lra diag(H_*\cP)$ is a cofibration. \

We will also want to know how this resolution interacts with extension-restriction of scalars (see section~\ref{extensionandrestrictionofscalars}). Given a multifunctor $F:\cP\rightarrow\Q$, a $\cP$-$\Q$-bimodule $Y$ determined by $F$, and $X:=\cP$, the multicategory $\cP$ considered as a $\cP$-$\cP$-bimodule over itself, then the bar complex $B(X,\cP, Y)=X\circ\cP^{\circ n}\circ Y$ is a simplicial object in the category of $\cP$-$\Q$-bimodules together with the augmentation map\begin{equation*}\eta:\textrm{diag}(B(X,\cP,Y))\longrightarrow F_*(X).\end{equation*} 

\begin{corollary}\label{Barcomplexcylinder}The bar complex $\textrm{diag} B(X,\cP,Y)$ is cofibrant as a $\cP$-$\Q$-bimodule. Moreover, the augmentation map\begin{equation*}\eta:\textrm{diag}B(X,\cP,Y)\longrightarrow F_*(X)\end{equation*}is a weak equivalence in the category of $\cP$-$\Q$-bimodules.\end{corollary}

\subsection{The $W$-construction}The main idea behind the Boardman-Vogt $W$-construction is to enrich the free operad construction by assigning lengths to edges in trees. The composition$$\mathfrak{F}(\cP)\cof W(\cP)\we \cP$$is identified with the counit of the free-forgetful adjunction between pointed collections and operads (See, for example,~\cite[Theorem 4.2]{EM},~\cite[Section 3]{BM07}).  If the collection underlying $\cP$ is cofibrant and well-pointed, then the counit $\mathfrak{F}(\cP)\lra \cP$ can be factored into a cofibration $\mathfrak{F}(\cP)\cof W(\cP)$ followed by a trivial fibration $ W(\cP)\trfib \cP$. Since $\mathfrak{F}(\cP)$ is a cofibrant operad, the $W$-construction provides a \emph{cofibrant resolution} for $\cP$~\cite[5.1]{BM06}.\

We also have the notion of a \emph{relative} $W$-construction, which resolves a morphism between multicategories $u:\cP\lra\Q$. This relative version produces an object $W(\Q)_{\cP}$ which is characterized by the property that algebras over this operad satisfy the operations from $\Q$ up to coherent homotopy, while they satisfy the operations from $\cP$ on the nose.\

\begin{example}Stasheff's $A_{\infty}$-operad can be obtained as the relative Boardman-Vogt resolution $W(I_*\to \mathbb{A})$ where $I_*$ is the operad for pointed objects.\end{example}


While most things work in multicategories the same as they do for the more classical operads, we do have to keep track of objects when doing the $W$-construction. More explicitly, given a map $\alpha:D\lra C$ between sets of objects, we can consider the adjunction $\alpha_*:\multi_D\leftrightarrows\multi_C:\alpha^*$ between multicategories with $D$-objects and multicategories with $C$-objects. For $\Q\in\multi_{D}$ and $\cP\in\multi_{C}$, there exist natural maps$$\alpha_*W(\Q)\lra W(\alpha_*\Q)\text{ and } W(\alpha^*\cP)\lra\alpha^* W(\cP),$$but in general these maps are not isomorphisms. If $\alpha$ is injective, we know that there is an explicit description of $\alpha_*(Q)$, as$$\alpha_*(Q)(d_1,\dots,d_n;d)=\begin{cases}Q(c_1,\dots,c_n;c)&\text{if }d_i=\alpha(c_i),d=\alpha(c),\\I&\text{if }n=1,d=d_1\not\in Im(\alpha),\\0&\text{otherwise},\end{cases}$$and, using this description, it is easy to show that the map $\alpha_*W(Q)\lra W(\alpha_*Q)$ is an isomorphism.\

Consider a cofibration $u:\cP\lra\Q$ between operads. In~\cite[Appendix]{BM03} they construct what is called the free extension $\cP[u]$ of $\cP$ by $u$. This free extension is determined by the universal property that operad maps out of $\cP[u]$ are in one-to-one correspondence with maps of collections out of $\Q$, whose restriction to $\cP$ (along $u$) is an operad map. In particular, the identity on $\Q$ induces a factorization of $u$ into maps $\cP\ra \cP[u]\ra\Q$. These maps, in turn, factor into cofibrations $\cP\cof \cP[u]\cof W(\Q)_{\cP}$ followed by a weak equivalence $W(\Q)_{\cP}\we \Q$ in such a way that the operad $W(\cQ)_{\cP}$ is a quotient of the operad $W(\Q)$. \

While this factorization always exists~\cite[Theorem 4.1]{BM07}, if we consider a category of operads which is a \emph{left proper model category}, then we may define the required factorization simply by taking a pushout: $\cP\lra\cP\cup_{W(\cP)} W(\Q)\we \Q.$ What's more, the object $W(\Q)_{\cP}$ is constructed as a sequential colimit of trivial cofibrations of collections:$$W_0(\Q)_{\cP}\trcof W_1(\Q)_{\cP}\trcof W_2(\Q)_{\cP}\trcof \cdots$$For each $k$, $W_k(\Q)_{\cP}$ is a quotient of $W_k(\Q)$, which is the piece of the operad $W(\Q)$ restricted to operations with inputs $\le k$. In other words, $W(\Q)_{\cP}$ is a quotient of $W(\cQ)$ by a filtration-preserving map.\

\section{Algebra Structures}It is well known to the experts that the category of $\multi(\C)$ is a \emph{closed}, symmetric monoidal category with respect to the Boardman-Vogt tensor product. By closed, we mean that there exists an internal hom-object $\underline{\Hom}$ satisfying the adjunction relation$$\multi(\cP\otimes_{BV}\Q,\R)\cong\multi(\cP,\underline{\Hom}(\Q,\R)).$$ 

This internal hom-object $\underline{Hom}(\mathcal{P},\mathcal{Q})$ is a multicategory with objects the multifunctors $\mathcal{P}\longrightarrow\mathcal{Q}$, and whose operations are type of multi-natural transformation. 

\begin{definition}\label{Homs}~\cite[Definition 2.2]{EM} For notational convenience, we denote the sequence $c_1,\dots,c_k$ as $\{c_i\}_{i=1}^k$. Given symmetric multicategories $\cP$ and $\Q$, we define $\underline{\Hom}(\cP,\Q)$ to be a multicategory with objects the multifunctors from $\cP$ to $\Q$. Given a sequence of multifunctors $F_1,\ldots,F_k:\cP\rightarrow\Q$ of multifunctors and a target multifunctor $G:\cP\rightarrow\Q$, we define a \emph{$k$-natural transformation} from $F_1,\ldots,F_k$ to $G$ to be a function $\xi$ that assigns to each object $a$ of $\cP$ a $k$-operation $\xi_a:(F_1a,\ldots,F_ka)\to Ga$ of $\Q$, such that for any $m$-ary operation $\phi:(a_1,\ldots,a_m)\to b$ in $\cP$, the following diagram commutes:
$$\xymatrix{
\{\{F_ja_i\}_{j=1}^k\}_{i=1}^m\ar[r]^-{\{\xi_{a_i}\}}
\ar[d]_-{\cong}
&\{Ga_i\}_{i=1}^m\ar[dd]^-{G\phi}
\\ \{\{F_ja_i\}_{i=1}^m\}_{j=1}^k \ar[d]_-{\{F_j\phi\}}
\\ \{F_jb\}_{j=1}^k \ar[r]_-{\xi_b}
&Gb.}$$The unlabelled isomorphism is the standard block permutation that shuffles $m$ blocks of $k$ entries each into $k$ blocks of $m$ entries each. The $k$-natural transformations form the $k$-ary operations in the multicategory $\underline{\Hom}(\cP,\Q)$. Composition and symmetric actions are induced by the composition and symmetric actions in $\Q$.\end{definition} 

In particular, if we restrict to the linear operations, the object $[\underline{\Hom}]_{1}$, gives $\multi$ an enrichment over the category of small categories. For any two multicategories $\cP$ and $\Q$, there exists a tensor product multicategory $\cP\otimes_{BV}\Q$ and a \emph{universal} bilinear map $(\cP,\Q)\to \cP\otimes_{BV}\Q$. This tensor product makes $\multi$ into a symmetric monoidal category. We will only give a brief description of the construction now, but refer the reader to the highly readable version in~\cite{EM}. 

\begin{definition}~\cite{EM}\label{bilin} Let $\cP$, $\Q$, and $\R$ be multicategories. A \emph{bilinear map}$$f:(\cP,\Q)\to \R$$ consists of the following data:\begin{enumerate}
\item A function $f:\obj (\cP)\times\obj(\Q)\to\obj(\R)$,
\item For each $m$-ary operation $\phi\in\cP(a_1,\ldots,a_m; a)$ of $\cP$ and each object $b$ of $\Q$, an $m$-ary operation $f(\phi,b)\in\R(f(a_1,b),\ldots,f(a_m,b);f(a,b))$ of $\R$,
\item For each $n$-ary operation $\psi\in\Q(b_1,\ldots,b_n;b)$ of $\Q$ and object $a$ of $\cP$, an $n$-operation $f(a,\psi)\in\R(f(a,b_1),\ldots,f(a,b_n);f(a,b))$ of $\R$\end{enumerate}such that:
\begin{enumerate}
\item For each object $a$ of $\cP$, $f(a,-)$ is a multifunctor from $\Q$ to $\R$,
\item For each object $b$ of $\Q$, $f(-,b)$ is a multifunctor from $\cP$ to $\R$,
\item Given an $m$-operation $\phi\in\cP(a_1,\ldots,a_m;a)$ and an $n$-operation $\psi\in\Q(b_1,\ldots,b_n;b)$ in $\Q$, the following diagram commutes:
$$\xymatrix{
\{\{f(a_i,b_j)\}_{i=1}^m\}_{j=1}^n \ar[r]^-{\{f(\phi,b_j)\}}
\ar[d]_-{\cong}
&\{f(a,b_j)\}_{j=1}^n\ar[dd]^-{f(a,\psi)}
\\ \{\{f(a_i,b_j)\}_{j=1}^n\}_{i=1}^m \ar[d]_-{\{f(a_i,\psi)\}}
\\ \{f(a_i,b)\}_{i=1}^m\ar[r]_-{f(\phi,b)}&f(a,b).}$$\end{enumerate}

The set of bilinear maps is denoted as $\textrm{Bilin}(\cP,\Q;\R)$.\end{definition}

A multifunctor $\cP\times \Q\to \R$ assigns a $k$-operations in $\R$ to each pair of $k$-operations from $\cP$ and $\Q$. On the other hand, a bilinear map assigns an $m\times n$-operation in $\R$, to each pair $(\phi,\psi)$, where $\phi$ is and $m$-operation $\cP$ and $\psi$ is an $n$-operation $\Q$. When restricted to the linear operations, a bilinear map $f:(\cP,\Q)\to \R$ is precisely a functor $[f]_1:[\cP]_1\times[\Q]_1\to [\R]_1$ of the underlying categories. Objects of $\textrm{Bilin}(M,N;P)$ are the objects of a multicategory naturally isomorphic to both $\underline{\Hom}(\cP,\underline{\Hom}(\Q,\R))$ and $\underline{\Hom}(\Q,\underline{\Hom}(\cP,\R))$.

\subsection{Construction of $\otimes_{BV}$}Let $\cP$ and $\Q$ be fixed multicategories, and construct the coproducts of multicategories$$\coprod_{a\in\obj(\cP)}\Q\quad\hbox{and}\quad\coprod_{b\in\obj(\Q)}\cP.$$ The coproduct is a universal morphism, i.e. given a multicategory $\R$, then $\coprod_{a\in\obj(\cP)}\Q$ is the universal source for any multifunctor which maps the objects $\obj(\cP)\times\obj(\Q)$ to $\obj(\R)$ and is a multifunctor with respect to $\Q$. Similarly, $\coprod_{b\in\obj(\Q)}\cP$ is universal for maps that send objects of $\obj(\cP)\times\obj(\Q)$ to $\obj(\R)$ which are multifunctors in $\cP$. If we are given a bilinear map $f:(\cP,\Q)\to\R$, it follows that we have multifunctors from both $\coprod_{a\in\obj(\cP)}\Q$ and $\coprod_{b\in\obj(\Q)}\cP$ to $\R$. Therefore the bilinear map $f$ induces a map from the pushout
$$\xymatrix{
\mathfrak{F}(\obj(\cP)\times\obj(\Q))\ar[r]\ar[d]
&\displaystyle\coprod_{b\in\obj(\Q)}\cP\ar[d]
\\ \displaystyle\coprod_{a\in\obj(\cP)}\Q\ar[r]
&\cP + \Q,}$$ to the multifunctor $\R$. The object in the upper lefthand corner is the free symmetric multicategory on the set of objects $\obj(\cP)\times\obj(\Q)$. It follows that the pushout $\cP +\Q$ is universal with respect to maps that are multifunctors in each variable separately.\

The $BV$-tensor product is the quotient $\cP + \Q$ after we force the bilinearity relations to commute. We construct this quotient as follows. For each $m$-ary operation $\phi\in\cP(a_1,\ldots,a_m;a)$ in $\cP$ and each $n$-ary operation $\psi\in\Q(b_1,\ldots,b_n; b)$ in $\Q$, define two non-symmetric collections $X(\phi,\psi)$ and $Y(\phi,\psi)$. The object sets of both $X$ and $Y$ will be the set $$(\{a_1,\ldots,a_m\}\times\{b_1,\ldots,b_n\})\cup\{(a,b)\}.$$ We give $X(\phi,\psi)$  precisely two operations, both with source $\{\{(a_i,b_j)\}_{i=1}^m\}_{j=1}^n$ and target $(a,b)$. We give $Y(\phi,\psi)$ exactly one operation with source $\{\{(a_i,b_j)\}_{i=1}^m\}_{j=1}^n$ and target $(a,b)$. We then define a map of collections $X(\phi,\psi)\lra Y(\phi,\psi)$ which sends the two operations of $X$ to the unique operation of $Y$.  We define a second map of collections from $X(\phi,\psi)$ to the underlying collection of $\cP +\Q$, denoted $U(\cP + \Q)$, by sending each operation of $X$ one way around the diagram $$\xymatrix{
\{\{(a_i,b_j)\}_{i=1}^m\}_{j=1}^n \ar[r]^-{\{(\phi,b_j)\}}
\ar[d]_-{\cong}
&\{(a,b_j)\}_{j=1}^n\ar[dd]^-{(a,\psi)}
\\ \{\{(a_i,b_j)\}_{j=1}^n\}_{i=1}^m \ar[d]_-{\{(a_i,\psi)\}}
\\ \{(a_i,b)\}_{i=1}^m\ar[r]_-{(\phi,b)}&f(a,b).}$$

Then apply the free, non-symmetric multicategory functor to these collections and form the following pushout:
$$\xymatrix{
\displaystyle\coprod_{(\phi,\psi)}FX(\phi,\psi)\ar@<1ex>[r]\ar[d]
&\mathop{FU(\cP + \Q)}\limits_{\phantom\phi}\ar[d]
\\ \displaystyle\coprod_{(\phi,\psi)}FY(\phi,\psi)\ar@<1ex>[r]
&\mathop{\cP \otimes_{BV} \Q}\limits_{\phantom\phi}}$$where $F(-)$ denotes the free non-symmetric multicategory functor(see appendix). This quotient is precisely what it means to force the diagrams in the definition of a bilinear map to commute, so $\cP\otimes_{BV}\Q$ is a universal bilinear target. We then go back and add the symmetric actions in a symstematic way. \

The \emph{unit} of the $BV$-tensor product is the multicategory $\mathcal{I}$ with one object and only the identity morphism on that object (see Example~\ref{unit}). \

\begin{remark}The symmetric actions are critical to the definition of the Boardman-Vogt tensor product. It is possible define a version of the tensor product on planar (a.k.a. non-symmetric) multicategories which forgets the bilinear relations. This is just $\cP +\Q$, which we called the \emph{coproduct of operads}. This does still form a \emph{closed} monoidal structure, but the internal hom-objects for this structure are not as well behaved as those presented above. In particular, one cannot define multilinear transformations between planar operads. We can still define transformations where the domain consists of a single multifunctor.\end{remark}

For a general symmetric monoidal category $\mathcal{C}$, the construction of the Boardman-Vogt tensor product still makes sense if either $\C$ is Cartesian closed or we restrict to \emph{Hopf} operads $\mathcal{P}$ and $\mathcal{Q}$. Hopf operads are characterized by the property that their algebra categories $\Alg_{\mathcal{P}}(\mathcal{C})$ and $\Alg_{\mathcal{Q}}(\mathcal{C})$ are again symmetric monoidal categories. In this case, the $BV$-tensor product tells us that a $(\mathcal{P}\otimes_{BV}\mathcal{Q})$-algebra in $\mathcal{C}$ is the same thing as a $\mathcal{P}$-algebra in $\Alg_{\mathcal{Q}}(\mathcal{C})$, and is also the same thing as a $\mathcal{Q}$-algebra in $\Alg_{\mathcal{P}}(\mathcal{C})$.\

\subsection{Derived Tensor Products}The Bordman-Vogt tensor product can be derived in the usual way, i.e. $$\cP\otimes^{\mathbb{L}}_{BV}\Q:= (\cP)_{c}\otimes_{BV}\Q.$$ Unfortunately, $\multi(\C)$ is not a monoidal model category, and $\otimes_{BV}$ does not preserve weak equivalences in general. It is true, however, that given a weak equivalence $\cP\lra\cP'$ which \emph{fixes} objects that $\cP +\Q\lra \cP' +\Q$ is a weak equivalence (see ~\cite{DH}, ~\cite{FV}). \

\subsection{Adjunction Relations}It can be helpful to think about these various monoidal structures in terms of generators and relations. Let $S$ be a fix set of operations of a multicategory $\cP$.  The multicategory $<S>$ \emph{generated by} $S$ is the smallest sub-multicategory of $\cP$ that contains all the operations in $S$. If $<S>=\cP$, we say that $\cP$ is \emph{generated by} $S$. Let $\cP$ and $\Q$ be multicategories, $\phi$ an operation of $\cP$ and $b$ an object of $\Q$.  Then we write $\phi\otimes_{bv} b$ for the operation of $\cP\otimes_{BV}\Q$ induced from $\phi$ and $b$ by the universal bilinear map $(\cP,\Q)\to \cP\otimes_{BV}\Q$.  Similarly, given an object $a$ of $\cP$ and an operation $\psi$ of $\Q$ we write $a\otimes_{bv}\psi$ for the operation of $\cP\otimes_{BV}\Q$ induced by $a$ and $\psi$. The universal property of the tensor product implies the following proposition.\

\begin{prop}\cite{EM, FZ} The operations $a\otimes_{bv}\psi$ and $\phi\otimes_{bv}b$ generate the multicategory $\cP\otimes_{BV}\Q$.\end{prop}Using this characterization, one can prove that we obtain the following adjunction$$\multi(\cP\otimes_{BV}\Q,\R)\cong\multi(\cP,\underline{\Hom}(\Q,\R))$$which enriches to a natural isomorphism of multicategories$$\underline{\Hom}(\cP\otimes_{BV} \Q,\R)\cong\underline{\Hom}(\cP,\underline{\Hom}(\Q,\R)).$$\

\begin{prop}\cite{EM}\label{tensorGen} The $k$-operations of $\underline{\Hom}(\cP\otimes_{BV}\Q,\R)$ are precisely those functions as in Lemma~\ref{natural} which are natural with respect to all morphisms of the form $a\otimes_{bv}\psi$ or $\phi\otimes_{bv} b$.\end{prop}

\begin{lemma}\cite{EM}\label{natural} Fix two multicategories $\cP$ and $\Q$, and suppose that $<S>$ is a generating set of operations for $\cP$. Then given a sequence of multifunctors $F_1,\dots,F_k,G:\cP\ra\Q$ and a map $\xi$ which assigns to each object $a$ of $\cP$ a $k$-operation $\xi_a:(F_1a,\dots,F_ka)\to Ga$ of $\Q$ such that the diagram
$$\xymatrix{
\{\{F_ja_i\}_{j=1}^k\}_{i=1}^m\ar[r]^-{\{\xi_{a_i}\}}
\ar[d]_-{\cong}
&\{Ga_i\}_{i=1}^m\ar[dd]^-{G\phi}
\\ \{\{F_ja_i\}_{i=1}^m\}_{j=1}^k \ar[d]_-{\{F_j\phi\}}
\\ \{F_jb\}_{j=1}^k \ar[r]_-{\xi_b}
&Gb}$$commutes for all $\phi$ in $<S>$. Then the diagram commutes for all operations of $\cP$, so $\xi$ is a $k$-natural transformation.\end{lemma}

\begin{proof}This is proved in~\cite{EM}, but we include the proof here because it is useful. First, we assume that we are given elements $\phi_1,\dots,\phi_n$ in $<S>$ with $\phi_i\in \cP(a_{i1},\dots,a_{im_i}; b_i)$ and $\psi\ in\cP(b_1,\dots,b_n;c)$.  Then the following diagram shows that our given transformation $\xi$ is natural with respect to the composition $\psi\circ(\phi_1,\dots,\phi_n)$ in $\cP$:
$$\xymatrix@C+25pt{
\{\{\{F_j(a_{is})\}_{s=1}^{m_i}\}_{i=1}^n\}_{j=1}^k
\ar[r]^-{\{\{F_j\phi_i\}_{i=1}^n\}_{j=1}^k} \ar[d]_-{\cong}
&\{\{F_jb_i\}_{i=1}^n\}_{j=1}^k \ar[r]^-{\{F_j\psi\}_{j=1}^k}
\ar[d]_-{\cong}
&\{F_jc\}_{j=1}^k \ar[ddd]^-{\xi_c}
\\ \{\{\{F_j(a_{is})\}_{s=1}^{m_i}\}_{j=1}^k\}_{i=1}^n
\ar[r]^-{\{\{F_j\phi_i\}_{j=1}^k\}_{i=1}^n} \ar[d]_-{\cong}
&\{\{F_jb_i\}_{j=1}^k\}_{i=1}^n
\ar[dd]^-{\{\xi_{b_i}\}_{i=1}^n}
\\ \{\{\{F_j(a_{is})\}_{j=1}^k\}_{s=1}^{m_i}\}_{i=1}^n
\ar[d]^-{\{\{\xi_{a_{is}}\}_{s=1}^{m_i}\}_{i=1}^n}
\\ \{\{G(a_{is})\}_{s=1}^{m_i}\}_{i=1}^n
\ar[r]^-{\{G\phi_i\}_{i=1}^n}
&\{Gb_i\}_{i=1}^n \ar[r]^-{G\psi} &Gc.}$$

Now, for every $\sigma\in\Sigma_n$, the following diagram shows that $\xi$ is natural with respect to the symmetric actions $\psi\cdot\sigma$:
$$\xymatrix@C+25pt{
\{\{F_jb_{\sigma(i)}\}_{i=1}^n\}_{j=1}^k \ar[r]^-{\cong}
\ar[d]^-{\{\sigma\}}&\{\{F_jb_{\sigma(i)}\}_{j=1}^k\}_{i=1}^n
\ar[r]^-{\{\xi_{b_{\sigma(i)}}\}_{i=1}^n} \ar[d]^-{\{\sigma\}}
&\{Gb_{\sigma(i)}\}_{i=1}^n \ar[d]^-{\{\sigma\}}
\\ \{\{F_jb_i\}_{i=1}^n\}_{j=1}^k \ar[r]^-{\cong}
\ar[d]^-{\{F_j\psi\}_{j=1}^k}
&\{\{F_jb_i\}_{j=1}^k\}_{i=1}^n \ar[r]^-{\{\xi_{b_i}\}_{i=1}^n}
&\{Gb_i\}_{i=1}^n \ar[d]^-{G\psi}
\\ \{f_jc\}_{j=1}^k \ar[rr]^-{\xi_c}&&Gc.}$$Since we know that $\xi$ is natural with respect to the generating operations, it now follows that $\xi$ is natural with respect to all morphisms in $\cP$. Therefore $\xi$ is a $k$-natural transformation.\end{proof}

The bijection on objects$$\multi(\Q,\underline{\Hom}(\cP,\R))\leftrightarrow\multi(\cP\otimes_{BV}\Q,\R)\leftrightarrow\multi(\cP,\underline{\Hom}(\Q,\R))$$ can be extended to functors between simplicial multicategories$$\underline{\Hom}(\Q,\underline{\Hom}(\cP,\R))\leftrightarrows\underline{\Hom}(\cP\otimes_{BV}\Q,\R)\leftrightarrows\underline{\Hom}(\cP,\underline{\Hom}(\Q,\R)).$$ In particular, before we take the quotient forcing the bilinear relations, we have a functor of simplicial categories$$U:[\underline{\Hom}]_{1}(\cP + \Q,\R)\longrightarrow[\underline{Hom}]_{1}(\cP,[\underline{\Hom}]_{1}(\Q,\R)).$$  

\begin{lemma}\label{enveloping}There exists a left adjoint to $U$ which we will denote by $E$.\end{lemma}

\begin{proof}This can be checked explicitly on generators and relations.\end{proof} 

In the next section we show that the adjoint pair $(E,U)$ can be easily extend an adjoint pair of functors $U:\cP +\Q\searrow\multi(\C)\leftrightarrows {}_{\cP}\M_{\Q}: E$, where ${}_{\cP}\M_{\Q}$ denotes the category of pointed $\cP$-$\Q$-bimodules. In this case, we will refer to $E$ as the \emph{enveloping functor}. The functor $E$ is left adjoint to a functor which preserves fibrations and weak equivalences, and thus $E$ preserves cofibrant objects and weak equivalences between cofibrant objects. We define the \emph{left derived functor} of $E$, denoted $\mathbb{L}E$, as $\mathbb{L}E(M):=E(M_{c})$, where $M_c$ is a cofibrant resolution of the object $M$ as a $\cP$-$\Q$-bimodule. Note that this left derived functor $\mathbb{L}E$ lands in the category $\cP +\Q\searrow\multi(\C)$ rather than the homotopy category.\

An operad $\cP$, like any monoid, can be considered as a $\cP$-$\cP$-bimodule over itself. An important property of $\mathbb{L}E$ is that $\mathbb{L}E(\cP)=\cP$, when the cofibrant resolution of $\cP$ we take is the Hochschild resolution, $H_*(\cP)$.

\begin{example}[Endomorphism Modules] For any two objects $X,Y$ in a symmetric monoidal category $(\M,\otimes, \unit_{M})$ we can define a collection whose $k$-operations are given by\begin{equation*}End_{X,Y}(k):=\M(X^{\otimes k},Y).\end{equation*}This collection can be given a $\Sigma_{k}$-action by permuting the source factors. We have natural composition products\begin{equation*}\circ_{i}:\M(X^{\otimes k},Y)\otimes \M(X^{\otimes l},X)\lra\M(X^{\otimes k+l-1},Y)\end{equation*}which implies that $\End_{X,Y}$ is a right module over the endomorphism operad $\End_{X}$. Given that $X$ is a $\Q$-algebra, i.e. that there exists an operad homomorphism $\alpha:\Q\ra\End_{X}$ then we can say that $\End_{X,Y}$ is a right $\Q$-module by restriction along the structure map.\ 

At the same time, we could consider the composition maps\begin{equation*}\M(Y^{\otimes r},Y)\otimes \M(X^{\otimes n_1}, Y)\otimes...\otimes\M(X^{\otimes n_r},Y)\lra\M(X^{\otimes n_1+...+n_r},Y)\end{equation*}which makes $\End_{X,Y}$ into a left module over the operad $\End_{Y}$. If $Y$ is a $\cP$-algebra, i.e. there exists an map $\beta:\cP\ra\End_{Y}$ then $\End_{X,Y}$ is a left $\cP$-module by restriction. \end{example} 

\begin{example}\label{PEndP}Consider an operad $\cP$ as a $\cP$-$\cP$-bimodule over itself. One way to do this is to recall that there exists a natural collection $\End_{\cP,\cP}:=\End_{\cP}$, which has two structure maps which are isomorphisms $\cP\ra\End_{\cP}$ which commute in a universal way. In other words, we could consider the $\cP$-$\cP$-bimodule structure on $\End_{\cP}$ as belonging to the space $\underline\Hom_{1}(\cP,\underline\Hom_{1}(\cP,\End_{\cP}))$. Once we apply the functor $E$, we  are consider $\End_{\cP}$ as a multicategory, together with a map from $\cP +\cP\lra\End_{\cP}$. This map is induced by the structure maps $\cP\ra\End_{\cP}$, which are isomorphisms, and thus $\cP +\cP\lra\End_{\cP}=\cP$ is just the fold map.\end{example}

\begin{lemma}The diagonal functor commutes with $\mathbb{L}E$. In particular, $E(\textrm{diag}(H_{*}\cP))$ is isomorphic to $\textrm{diag}(E(H_{*}\cP))$ in $(\cP +\cP)\searrow\multi(\C)$.\end{lemma} 

\begin{proof}The proof follows as in~\cite[5.3]{DH}, once we note that, by construction, the multicategories $E(\textrm{diag}(H_{*}\cP))$ and $\textrm{diag}(E(H_{*}\cP))$ have the same set of objects.\end{proof} 

\begin{prop}\label{Hochschildgood} The multicategory $\textrm{diag}E(H\cP)$ is a cofibrant object in $(\cP + \cP)\searrow\multi(\C)$. Moreover, there exists a weak equivalence of multicategories $\textrm{diag}(EH_*\cP)\lra\cP$.\end{prop}

This is a many objects version of ~\cite[5.4]{DH}. We prolong the functor $E$, applying $E$ to $H_*\cP$ degree-wise. Since in each degree $H_{n}\cP$ is a pointed bimodule (the basepoints come from the image of $H_0(\cP)$ under degeneracies), it follows that $E(H_n(\cP))$ is a simplicial object in the category $(\cP + \cP)\searrow\multi(\C)$ for each $n\ge0$. The functor $E$ commutes with composition and units, i.e. the face and degeneracy maps, and thus $E(H(\cP))$ is a simplicial object in $(\cP + \cP)\searrow\multi(\C)$.\

It follows that the multicategory $\textrm{diag}E(H_*\cP)$ is an object under $\cP + \cP$. Since we know that $E(\cP)$ is isomorphic to $\cP$ in the category $(\cP +\cP)\searrow\multi(\C)$, this implies that the augmentation map $\eta: \textrm{diag}(E(H_*\cP))\lra E(\cP)$ factors the the fold map $\cP + \cP\lra\cP$.\

\begin{proof}[Proof of Proposition]The key observation is that $E(\cP)$ is isomorphic to $\cP$ under $\cP+\cP$. It follows that $E(H_{n}\cP)=E(\cP^{\circ (n+2)})$ is isomorphic to the free multicategory on $U(\cP^{\circ n})$, the underlying collection of $\cP^{\circ n},$ together with a map from $\cP +\cP$ (coming from the basepoint). \end{proof} 

\subsection{Spaces of Algebra Structures}The category of right $\Q$-modules, $\M_{\Q}$, is a symmetric monoidal $\C$-category (see appendix) and, as such, it makes sense to define generalized $\cP$-algebras taking values in $\M_{\Q}$. More explicitly, a $\cP$-algebra structure on a right $\Q$-module $M$ is a multifunctor from $\cP$ to $\End_{\Q}(M)$ (see also,~\cite{FresseBook}). The endomorphism multicategory $\End_{\Q}(M)$ has $n$-ary operations given by$$\Hom_{\M_{\Q}}(M^{\otimes n}, M)$$which is the space of right $\Q$-module homomorphisms from the $n$-fold tensor product of $M$ to $M$. This forms a simplicial multicategory in the usual way, and has both a natural left $\End_{\Q}(M)$-action and a natural right $\Q$-action which makes $\End_{\Q}(M)$ into an $\End(M)$-$\Q$-bimodule. The $\cP$-$\Q$-bimodule structures are in one-to-one correspondence with $\cP$-algebra structures, i.e. operad homomorphisms $\cP\lra\End(M)$.\

\begin{example}The multicategory $\Q$ is naturally a right $\Q$-module over itself, and the left action of $\Q$ on itself gives an equivalence $\End_{\Q}(\Q)\cong\Q.$ \end{example}

For a fixed object $X$ in $\M$, the simplicial set $[\underline{\Hom}]_{1}(\cP,\End(X))$ is the \emph{space of $\cP$-algebra structures} on the object $X$. This is again due to the fact that the endomorphism operad of an object $X\in\M$ is the universal object in $\multi(\C)$ acting on $X$, i.e. that any action on $X$ by an object $\cP$ in $\multi(\C)$ is the restriction of the $\End(X)$-action on $X$ along a uniquely determined morphism $\cP\rightarrow\End(X)$.  

\begin{example}Let $\mathbb{C}$ denote the commutative operad and let $\widehat{\Sets}$ denote the underlying multicategory of $\Sets$, namely the $n$-ary operations of $\widehat{\Sets}$ are given by$$\Sets(x_1\times...\times x_n;x).$$ One can check straight from the definitions that $[\underline{\Hom}]_{1}(\mathbb{C},\widehat{\Sets})$ is isomorphic to the category of commutative monoids.\end{example}

\subsection{The Moduli Space of Algebra Structures}Fix a simplicial multicategory $\cP$ with object set $\obj(\cP)=S$. Let $\M$ be a symmetric monoidal category which is tensored and cotensored over $\sSets$, and let $\M^{S}$ denote the category obtained as the product of copies of the category $\M$ indexed over $S$. If $\M$ is a (simplicial) monoidal model category, the the product category $\M^{S}$ inherits a (simplicial) model structure from the model structure on $\M$, where the fibrations, cofibrations, and weak equivalences formed coordinatewise. The simplicial category  $[\underline{\Hom}(\cP,\M)]_{1}$ is a category of algebras over a \emph{triple}, or \emph{monad}, $T$. Explicitly, if we fix an object $x$ of $\cP$, and let $A$ be an object in $\M^{S}$, then we have a triple$$T:=(T(A))_{x}=\amalg_{n\geq 0}(\amalg_{x_{1},\dotsc,x_{n}\in \obj(\cP)}\cP(x_{1},\dotsc,x_{n};x)\otimes_{\Sigma_{n}} (A(x_{1})\otimes... \otimes A(x_{n})),$$let $\eta:A\to T(A)$ be the map$$A(x)\lra\{\id_{x}\}\otimes A_{x}\to \cP(x;x)\otimes A(x)\to (T(A(x)),$$ and let $\mu:TT(A)\to T(A)$ just be the map induced by the composition operations of $\cP$.\

\begin{remark}If we consider $T$ as a functor$$T:\M^{S}\lra[\underline{\Hom}]_{1}(\cP,\M),$$then one can easily check that $T$ is the left adjoint to the forgetful functor$$[\underline{\Hom}]_{1}(\cP,\M)\lra\M^{S}.$$In other words, $T(A)$ is precisely the free $\cP$-algebra on $A:=\{A(x)|A(x)\in\M\}_{x\in\obj(\cP)}$. Denote the free $\cP$-algebra on $A$ by $F_{\cP}(A)$.\end{remark}

The following theorem is an easy generalization of~\cite[Theorem 11.2]{EM}, which is itself a generalization of~\cite{May}. 

\begin{theorem}\cite{EM}\cite{May}Given the triple $T$ above on the category $\M$, a $T$-algebra structure on an object of $\M$ is equivalent to a simplicial multifunctor from $\cP$ to $\M$, and the simplicial category of $T$-algebras is isomorphic to the simplicial category $[\underline{\Hom}]_{1}(\cP,\M).$\end{theorem} 

\begin{corollary}\label{bicomplete}The category $[\underline{\Hom}]_{1}(\cP, \M_{\Q})$ is a symmetric monoidal category over $\C$ and has all small limits and colimits.\end{corollary}We know that the category of right $\Q$-modules $\M_{\Q}$ admits a cofibrantly generated monoidal model category structure over $\sSets$. In this case, a map of right $\Q$-modules is a cofibration of right $\Q$-modules if, and only if it, is a retract of a relative $I$-complex where$$I:=\{K\circ\Q\longrightarrow L\circ\Q|n\ge0\}$$and $K\rightarrow L$ runs over the generating cofibrations of $\coll(\C)$. A map of right $\Q$-modules is an acyclic cofibration if, and only if, it is a retract of a relative $J$-complex, where$$J:=\{K\circ\Q\longrightarrow L\circ\Q|n\ge0\}$$ and $K\rightarrow L$ runs over the generating acyclic cofibrations of $\coll(\C)$ (for more on the model structure of $\coll(\C)$, see~\cite{BM07}).\

The generating (acyclic) cofibrations for the product category $\M_{\Q}^{S}$ are defined similarly. Explicitly, for a fixed object $x$ in $\cP$ we let $\iota_{x}:\M_{\Q}\longrightarrow\M_{\Q}^{S}$ be the left adjoint to the evaluation functor $Ev_{x}:\M_{\Q}^{S}\longrightarrow\M_{\Q}$, i.e. given a fixed right $\Q$-module $A$ and an arbitrary object $y$ in $\cP$, the object $(\iota_{x}A)_{y}$ in $\M_{\Q}^{S}$ is either $A$ if $x=y$ or trivial otherwise. Now, we can define the sets$$\iota_{*}I:=\{\iota_{x}f|f\in I, x\in\obj(\cP)\}$$and$$\iota_{*}J:=\{\iota_{x}f|f\in J,x\in\obj(\cP)\}.$$So, a map is a cofibration of $\M_{\Q}^{S}$ if, and only if, it is a retract of $\iota_{*}I$; a map of $\M_{\Q}$ is an acyclic cofibration of $\M_{\Q}^{S}$ if, and only if, it is a retract of $\iota_{*}J$. 

\begin{theorem}[Model Structure]Let $\cP$ and $\Q$ be two simplicial multicategories. If the category of right $\Q$-modules admits a cofibrantly generated model category structure then the category $[\underline{\Hom}]_{1}(\cP,\M_{\Q})$ admits a cofibrantly generated model category structure where a morphism is a weak equivalence (respectively, fibration) if the underlying map of right $\Q$-modules is a weak equivalence (respectively, fibration).\end{theorem}

We will defer the proof of the theorem to the appendix. 

\begin{definition}A right $\Q$-module $M$ is \emph{pointed} if it comes with a unit map$$\Q=\unit\circ\Q\longrightarrow M\circ\Q\longrightarrow M.$$ A pointed right $\Q$-module $M$ will be called \emph{quasi-free} if the natural map$$\Q\lra\End_{\Q}(M)$$is weak equivalence of right $\Q$-modules. We will say that a $\cP$-$\Q$-bimodule is pointed (respectively, \emph{right quasi-free}) if it is pointed (respectively, \emph{quasi-free}) as a right $\Q$-module.\end{definition}
The category of \emph{pointed} $\cP$-$\Q$-\emph{bimodules} is equivalent to the category of $\cP$-$\Q$-bimodules $M$ which have a natural homomorphism of bimodules $F_{\cP}(\Q)\lra M$. Moreover, the category of pointed $\cP$-$\Q$ bimodules admits a cofibrantly generated model structure~\cite[Chapter 14]{FresseBook}. 

\begin{remark}\label{pointedisreduced}In Fresse~\cite[Chapter 14]{FresseBook} the model structure on $\cP$-$\Q$-bimodules depends on the multicategory $\Q$ being \emph{reduced}. We claim this is equivalent to the condition of our bimodules being \emph{pointed}. Notice that endomorphism-operads are not reduced, since the zero operations $\End(A)(-;A(x))= A(x)$. However, any object $A$ \emph{under} $\unit$ defines a \emph{reduced} endomorphism multicategory $\widetilde{\End}(A)$. If $\cP$ is reduced, a $\cP$-algebra structure on $A$ is also equivalent to a base point $\unit\rightarrow A$ together with an operad map $P\rightarrow\widetilde{\End}(A)$.\end{remark}

\begin{corollary}The simplicial category $[\underline{\Hom}]_{1}(\cP,\M_{\Q})$ is simplicially Quillen equivalent to the category of pointed $\cP$-$\Q$-bimodules ${}_{\cP}\M_{\Q}$.\end{corollary} 

\section{Morita Theory}\label{Moritatheory}We will denote the category of right quasi-free $\cP$-$\Q$-bimodules by ${}_{\cP}\M^*_{\Q}.$ The model structure in the previous section clearly induces a model structure on${}_{\cP}\M^*_{\Q}.$ This section is devoted to showing that the derived mapping space $\Map^{h}(\cP,\Q)$ is weakly homotopy equivalent to the moduli space of quasi-free $\cP$-$\Q$-bimodules. 

\subsection{Extension and Restriction of Scalars}\label{extensionandrestrictionofscalars} Given a multicategory $\mathcal{P}$ with $\obj(\cP):=S$ and a map of sets, $F_0:T\longrightarrow S$, We can construct a multicategory $F^{*}(\mathbb{\mathcal{P}})$ with object set $T$, with operations given by\begin{equation*}F^{*}(\mathcal{P})(d_{1},\cdots,d_{n};d):=\mathcal{P}(Fd_{1},...,Fd_{n};Fd).\end{equation*} It follows that given a multifunctor $\psi:\R\lra\mathcal{S}$, we can show that $\R$ operates on any right $\mathcal{S}$-module $N$ through the morphism $\psi:\R\lra\mathcal{S}.$ This $\cR$-action defines a right $\R$-module $\psi^*N$ associated to $N$ by restriction. This structure has a natural left adjoint, which defines a right $\mathcal{S}$-module $\psi_*M$ by extension. The following results easily generalize from operads to general multicategories. 


\begin{prop}\label{extensionrestrictionrightmodules}Let $\psi:\R\rightarrow\mathcal{S}$ be a multifunctor. The extension-restriction functors$$\psi_*:\M_{\R}\longrightarrow\M_{\mathcal{S}}:\psi^*$$are functors of symmetric monoidal categories over $\C$ and define an adjunction relation in the $2$-category of symmetric monoidal categories over $\C$.\end{prop}

Given another multifunctor $\phi:\cP\rightarrow\Q$, we have extension and restriction functors on algebra categories\begin{equation*}\phi_*:{\Alg}_{\cP}(\M)\rightleftarrows{\Alg}_{\Q}(\M):\phi^*.\end{equation*}By definition, the restriction functor $\phi^*:\Alg_{\Q}(\M)\rightarrow\Alg_{\cP}(\M)$ reduces to the identity functor $\phi^*(B) = B$ if we forget operad actions. 

\begin{lemma} The extension functor $\phi^*$ preserves weak equivalences and fibrations.\end{lemma} 

In particular, the extension functor $\phi^*$ is the right adjoint in a Quillen adjunction. We delay the proof of the following proposition to the appendix. 

\begin{prop}\label{extensionrestrictionofalgebras}Let $\phi:\cP\rightarrow\Q$ be a multifunctor between two admissible, $\Sigma$-cofibrant simplicial multicategories, and let $\M$ be a left proper, cofibrantly generated monoidal model category over $\C$. Them \begin{equation*}\phi_*: \Alg_{\cP}(\M)\rightleftarrows\Alg_{\Q}(\M):\phi^* \end{equation*}defines a Quillen adjunction. Furthermore, if $\phi:\cP\rightarrow\Q$ is a weak-equivalence in $\multi(\C)$, then $(\phi_*,\phi^*)$ defines a Quillen equivalence.\end{prop}



In particular, extensions and restrictions on the left commute with extensions and restrictions on the right up to \emph{coherent functor isomorphisms}. 

\begin{prop}\label{changeofenrichingcat}Let $\cP$ be any well pointed, $\Sigma$-cofibrant multicategory enriched in $\C$. Given $\rho_*: \M\rightleftarrows\N: \rho^*$ a Quillen adjunction of monoidal model categories over $\C$. The functors \begin{equation*} \rho_*:\Alg_{\cP}(\M)\rightleftarrows\Alg_{\cP}(\N): \rho^*\end{equation*}induced by $\rho_*$ and $\rho^*$ define a Quillen adjunction. If $\rho_*:\M\rightleftarrows\N:\rho^*$ forms a Quillen equivalence, then \begin{equation*}\rho_*:\Alg_{\cP}(\M)\rightleftarrows\Alg_{\cP}(\N):\rho^*\end{equation*}forms a Quillen equivalence.\end{prop}

\subsection{Lifting Extension-Restriction of Scalars}Any multifunctor $\psi:\R\rightarrow\cS$ also induces a map $G:\cP+\R\lra\cP+\cS$ and a Quillen adjunction$$G_*:(\cP+\R)\searrow\multi(\C)\leftrightarrows(\cP+\cS)\searrow\multi(\C):G^*.$$The coproduct of operads preserves weak equivalences (see~\cite[4.3]{DH}~\cite{FZ}). Therefore, if $\psi$ is a weak equivalence which \emph{fixes objects}, then by~\cite[Prop. 2.5]{Rezk2}, we know that$$G_*:(\cP+\R)\searrow\multi(\C)\leftrightarrows(\cP+\cS)\searrow\multi(\C):G^*$$is a Quillen equivalence. This depends heavily on the fact that categories of simplicial multicategories with fixed sets of objects is a left proper model category (combine~\cite[4]{Rezk2} with~\cite[1.5]{BM07}).\

\begin{prop}\label{morita1}Let $\cP$ be an admissible, $\Sigma$-cofibrant multicatgory. Let $\psi:\R\rightarrow\mathcal{S}$ be a multifunctor between two locally cofibrant multicategories which fixes objects. If $\psi$ is a weak equivalence, then$$\mathbb{L}E(\psi):(\cP+\R)\searrow\multi(\C)\longrightarrow(\cP+\mathcal{S})\searrow\multi(\C)$$is a weak equivalence.\end{prop}

\begin{proof}Let $(\psi_{*},\psi^*)$ and $(G_*,G^*)$ be the Quillen equivalences above induced by $\psi.$ We want to show that given $B$ in ${}_{\cP}\M_{\R}$, $\mathbb{L}E(B)$ is weakly equivalent to $\mathbb{L}E(A)$ for some $A$ in ${}_{\cP}\M_{\cS}$. If $A$ is cofibrant, then $\psi_{*}(A)$ is a cofibrant object weakly equivalent to $B$. Since $E$ is left adjoint to a functor that preserves fibrations and weak equivalences, we know that $E$ preserves cofibrant objects and weak equivalences between cofibrant objects, and thus that $E\psi_*(A)$ is weakly equivalent to $\mathbb{L}E(B)=E(B_{c})$. The adjoint functor theorem implies that $G_*\mathbb{L}E(B)=E\psi_*(A)$. Since $G_*$ is the left adjoint of a Quillen equivalence, $G_*$ takes weak equivalences between cofibrant objects to weak equivalences, and we are done.\end{proof} 

\begin{lemma}\label{morita2}Let $\cP$ be an admissible, $\Sigma$-cofibrant simplicial multicategory and let $\Q$ be locally cofibrant. Let $f:M\ra N$ be a homomorphism between cofibrant pointed $\cP$-$\Q$-bimodules. Assume that $\mathbb{L}E(M)$ is an object of $(\cP + \Q)\searrow\multi(\C)$ so that the natural map $\Q\ra(\cP +\Q)\ra \mathbb{L}E(M)$ is a weak equivalence. Then if $f$ is a weak equivalence, the natural map $\Q\ra(\cP +\Q)\ra \mathbb{L}E(N)$ is also a weak equivalence.\end{lemma} 

\begin{prop}\label{mortia3}Let $F:\cP\lra\Q$ be a multifunctor, and let $\cP$ be admissible $\Sigma$-cofibrant and $\Q$ be locally cofibrant. Then $F$ equips $\cP$ with the structure of a right quasi-free $\cP$-$\Q$-bimodule by restriction on the right. Moreover, the induced operad $\mathbb{L}E(F)$ is an object of $(\cP+\Q)\searrow\multi(\C)$ so that the map $\Q\ra (\cP +\Q)\ra \mathbb{L}E(F)$ is a weak equivalence.\end{prop} 

\begin{proof}The first part follows from the isomorphism $\End_{\cP}(\cP)\cong\cP$. Let$$F^*:\Alg_{\cP}(\M_{\Q})\longrightarrow\Alg_{\cP}(\M_{\cP})$$be the restriction functor of proposition~\ref{changeofenrichingcat}, and let $F_*$ denote the left adjoint. It follows from standard arguments (~\cite[Before Theorem 4.1]{BM07}) that $F_*$ commutes with free algebras. It follows that $F_*$ takes our basepoint $F_{\cP}(\cP)\rightarrow\cP$ to the required basepoint $F_{\cP}(\Q)\rightarrow\Q$.\

Let $G:\cP+\cP\longrightarrow\cP+\Q$ be the map $id +F$ and let $(G_*,G^*)$ denote the induced Quillen adjunction on multicategories under $\cP +\Q$. Let $\eta:\textrm{diag}(H_*\cP)\lra\cP$ be the cofibrant resolution of $\cP$ in the category of $\cP$-$\cP$-bimodules we described in~\ref{Hochschildgood}. For the remainder of this proof, denote this cofibrant bimodule by $\cP_c$.\

Notice that both $\cP_c$ and $\cP$ are cofibrant objects at the level of right $\cP$-modules. This implies that induced map $j:F_{*}(\cP_c)\rightarrow F_{*}(\cP)=\Q$ is a weak equivalence at the level of right $\Q$-modules. Now, by Lemma~\ref{morita2}, we know that we may assume that $F_{*}(\cP_c)$ and $\Q$ are also fibrant objects (taking the fibrant replacement which fixes objects). It then follows from Theorem~\ref{invariance} that the $\cP$-algebra structure maps are compatible with this weak equivalence in a homotopy coherent way. It now follows that $EF_*(\cP_c)=G_*E(\cP_c)$.\

Now, using the Hochschild resolution as a cofibrant replacement for $\cP$, we have shown that $\mathbb{L}E(\cP)$ is precisely $\cP$ under the fold map $\cP+\cP\longrightarrow \mathbb{L}E(\cP)=\cP$. \

It remains to check two things: \begin{enumerate}
\item that the object $E(\cP_c)$ is a cylinder object for $E(\cP)$ and
\item that the composite\begin{equation}\xymatrix{\Q\ar[r]^{in_2} &\cP+\Q \ar[r]^{i} & E(F_{*}\cP_c)  \ar[r]^{E(\eta)}& \mathbb{L}E(F_*\Q)}\end{equation} is a weak equivalence.\end{enumerate}

To prove $(1)$, we consider what $E(P_c)$ looks like. $H_*(\cP)$ is a simplicial object in $\cP$-$\cP$-bimodules, so $E(H_*(\cP))$ is applied degree-wise. This means that $E(H_*(\cP))$ is a simplicial object under $\cP+\cP$. If we now apply the diagonal, we can consider the augmentation map $E(\eta):diag (E(H_*(\cP))\lra E(\cP)=\cP$, which, being a morphism in $(\cP+\cP)\searrow\multi(\C)$, is a factorization of the fold map\begin{equation*}\xymatrix{\cP+\cP \ar[r]^{i} & diag E(H_*(\cP))\ar[r]^{E(\eta)} & E(\cP)=\cP }.\end{equation*}

We know that $\textrm{diag}E(\cP_c))$ is isomorphic to $E(\cP_c)$ as simplicial multicategories under $\cP+\cP$. Moreover, we know that $\textrm{diag} E(\cP_c)$ is a cofibrant multicategory under $\cP+\cP$ and that the augmentation $\textrm{diag}E(\cP_c)\longrightarrow E(\cP)=\cP$ is a weak equivalence. Since $\textrm{diag}E(\cP_c)$ is cofibrant, it follows that the map$$i:\cP+\cP\lra\textrm{diag} E(\cP_c)$$is a cofibration. Putting this all together, we have\begin{equation}\xymatrix{\cP+\cP \ar[r]^{i} & E(\cP_c) \ar[r]_{E(\eta)} & E(\cP)=\cP }\end{equation}with $E(\eta)$ a weak equivalence and $i$ a cofibration.\ 

Now to show $(2)$, notice that $(1)$ implies that the composite map\begin{equation}\xymatrix{\cP \ar[r]^{in_{2}} & \cP+\cP  \ar[r]^{i}& E(\cP_c)}\end{equation}is a weak equivalence of multicategories. This follows from the $2$-out-of-$3$ property and the observation that \begin{equation}\xymatrix{\cP \ar[r]^{in_{2}} & \cP+\cP  \ar[r]^{i}& E(\cP_c)\ar[r]_{E(\eta)} & E(\cP)=\cP}\end{equation}is the identity map.\  

We restrict along $F$ to get the the map $\beta:\Q\rightarrow E(F_*P_c)$ which is the composite\begin{equation}\xymatrix{\Q\ar[r]^{in_2} & \cP+\Q\ar[r]^{i} & E(F_*\cP_c)}.\end{equation}Consider the diagram\begin{equation}\xymatrix{\Q \ar[r]^{\beta} & EF_*(\cP_c)\ar[r]^{E(\eta)} \ar[d]^{U} &  EF_*(\cP) \ar[d]^{U} \\ &F_*(\cP_c)\ar[r]^{j} & F_*(\cP).}\end{equation}

It follows that $\beta$ is a weak equivalence, as we are looking at the identity map $\Q\lra F_*(\cP)=Q$ factored by $\beta$ and a weak equivalence. \end{proof} 

Putting together Lemma~\ref{morita2}, Proposition~\ref{morita1} and Proposition~\ref{morita3} we have now shown that right quasi-free $\cP$-$\Q$-bimodules get lifted to zig-zags $[\cP\lra M\bwe \Q]$.\

\begin{theorem}\label{MoritaTheorem}If $M$ is right quasi-free $\cP$-$\Q$-bimodule then the natural map $\Q\ra\cP +\Q\ra \mathbb{L}E(M)$ is a weak equivalence in $(\cP+\Q)\searrow\multi(\C)$.\end{theorem}

\begin{proof}Let $M$ be a right quasi-free $\cP$-$\Q$-bimodule. We may assume that $M$ is fibrant and cofibrant by Lemma~\ref{morita2} and that the morphism $F_{\cP}(\Q)\longrightarrow M$ is a cofibration. The model structure on ${}_{\cP}\M^*_{\Q}$ implies that we still have a cofibration after forgetting the $\cP$-algebra structure $U_{\cP}F_{\cP}(\Q)\longrightarrow U_{\cP}(M).$\

Consider the square of right $\Q$-modules:
\begin{equation}
\xymatrix{
\End_{\Q}(f)\ar[r] \ar[d] &  \End_{\Q}(\Q) \ar[d] \\
\End_{\Q}(M)\ar[r] & \End_{\Q}(\Q,M). 
}
\end{equation} By Theorem~\ref{invariance}, we know that we have weak equivalences $\End_{\Q}(f)\lra\End_{\Q}(M)$ and $\End_{\Q}(f)\lra\End_{\Q}(\Q)=\Q.$ Now, if we remember that $M$ is actually a $\cP$-$\Q$-bimodule, and thus we have a $\cP$-algebra structure map $\alpha:\cP\longrightarrow\End_{\Q}(M)$. As in the proof of Theorem~\ref{invariance}, we can lift the $\cP$-algebra structure on $\End_{\Q}(M)$ to a $\cP$-algebra structure on $\End_{\Q}(f)$ in a homotopy coherent way. Moreover, we can compose to get the following right quasi-free $\cP'$-$\Q$-bimodule$$\beta:\cP\longrightarrow\End_{\Q}(f)\longrightarrow\End_{\Q}(\Q)=\Q.$$ By the proposition~\ref{morita3}, we know that $\Q\ra\cP+\Q\ra\mathbb{L}E(\beta)$ is a weak equivalence by proposition~\ref{morita1} we know that this implies that $\Q\ra\cP+\Q\ra\mathbb{L}E(\alpha)$ is a weak equivalence. \end{proof}



\subsection{Equivalences of Moduli Spaces}The main theorem of this paper is the following. 

\begin{theorem}The derived mapping space $\Map^h(\cP,\Q)$ is weakly homotopy equivalent to the moduli space of right quasi-free $\cP$-$\Q$-bimodules, ${}_{\cP}\M^*_{\Q}$.\end{theorem} 

The adjoint pair $(E,U)$ is a Quillen pair, in which $U$ preserves all weak equivalences. Before we can prove the main theorem we assume that $\alpha:\cP\lra\End_{\Q}(M)$ is a fixed $\cP$-$\Q$-bimodule structure, and consider the homotopy fiber of $U:\cP+\Q\searrow\multi(\C)\lra {}_{\cP}\M^*_{\Q}$ at the basepoint $M$.\

Let $\mathcal{B}$ be the subcategory of $\multi(\C)$ consisting of simplicial multicategories with fixed object sets $\obj(\cP)\times\obj(\Q)$. In this case, $\mathcal{B}$ is a left proper model category (see,~\cite{Rezk2}). For this next proposition, we consider the restriction of the adjoint pair $(E,U)$ to $\mathcal{B}$.\

\begin{prop}\label{hofiber}Suppose that $(E,U)$ is the Quillen pair above, and that the right adjoint $U,$ preserves all weak equivalences. Then if either:\begin{enumerate}
\item $M$ in ${}_{\cP}\M^*_{\Q}$ is cofibrant, or
\item $\mathcal{B}$ is a left proper model category and $E(M_c)\lra E(M)$ is a weak equivalence in $\mathcal{B}$ for all $M$ in ${}_{\cP}\M^*_{\Q}$
\end{enumerate}the homotopy fiber of $wU$ over $M$ in $\mathcal{B}$ is equivalent to the nerve of the under category $M\searrow wU$.\end{prop} 

\begin{proof}
The assumption that our right adjoint $U:\mathcal{B}\lra{}_{\cP}\M_{\Q}$ preserves weak equivalences, implies that the under category $M\searrow wU$ is isomorphic to the union of the components of the moduli space $w(EM\searrow \mathcal{B})$ containing maps $\alpha:EM\lra X_1$ whose adjoint $\alpha':M\lra UX_1$ is a weak equivalence in ${}_{\cP}\M^*_{\Q}$. Objects of the category $M\searrow wU$ look like $[M\we U(X_1)]$ and morphisms $[M\we U(X_1)]\ra [M\we U(X_2)]$ are weak equivalences $U(f):U(X_1)\we U(X_2)$ in ${}_{\cP}\M^*_{\Q}$ such that the obvious diagram commutes. Objects of $w(EM\searrow \mathcal{B})$ look like $[EM\we X_1]$ and morphisms $[EM\we X_1]\ra[EM\we X_2]$ are weak equivalences $f:X_1\we X_2$ in $\mathcal{B}$ such that the obvious diagram commutes.\

If we assume condition $(1)$, it follows from~\cite[After 2.5]{Rezk2} that given a weak equivalence $M\ra N$ between cofibrant objects in ${}_{\cP}\M^*_{\Q}$ that the induced map $M\searrow wU\lra N\searrow wU$ is a weak equivalence. If we instead assume condition $(2)$, then we know that the natural map $M\searrow wU\lra M_c\searrow wU$ is a weak homotopy equivalence by Rezk's characterization of left properness~\cite[2.5]{Rezk2}. \

Now, consider the subcategory of \emph{cofibrant} pointed $\cP$-$\Q$-bimodules $({}_{\cP}\M^*_{\Q})_c\subset {}_{\cP}\M^*_{\Q}.$ The objects of the category $w({}_{\cP}\M_{\Q})_{c}\searrow wU$ look like  $[M_1\we U(X_1)]$, where $M_1$ is a cofibrant pointed $\cP$-$\Q$-bimodule. Morphisms $[M_1\we U(X_1)]\lra [M_2\we U(X_2)]$ are pairs $(\phi,f)$ where $\phi:M_1\we M_2$ is a weak equivalence between cofibrant pointed $\cP$-$\Q$-bimodules and $U(f):U(X_1)\we U(X_2)$ is a weak equivalence (i.e. $f:X_1\ra X_2$ is a weak equivalence in $\mathcal{B}$). \

The category $w({}_{\cP}\M^*_{\Q})_{c}\searrow wU$ is a \emph{path space} construction. In particular, we have a commutative diagram:\begin{equation}\label{pathspace}\xymatrix{
w({}_{\cP}\M_{\Q})_{c}\searrow wU \ar[r]^{pr_2}\ar[d]^{pr_1} & w\mathcal{B} \ar[d]^{wU} \\
w({}_{\cP}\M_{\Q})_c \ar[r]_{in} & w({}_{\cP}\M_{\Q}) .}\end{equation} The map $pr_1$ takes the object $[M_1\we U(X_1)]$ to $M_1$ and $pr_2$ takes $[M_1\we U(X_1)]$ to $X_1$ in $\mathcal{B}$. The map $in$ is the induced inclusion of moduli spaces. Both the maps $in$ and $pr_2$ induce weak homotopy equivalences on moduli spaces.\

For a fixed cofibrant pointed $\cP$-$\Q$-bimodule $M$ we can construct a functor $M\searrow wU \lra M\searrow pr_2$, which is natural in $M$ and induces a weak homotopy equivalence on nerves. We want to show that the homotopy fiber of $pr_2$ over a cofibrant $M$ is weakly homotopy equivalent to the nerve of $M\searrow pr_2$. Now, we have already argued that every weak equivalence between (the cofibrant objects) $M_1\we M_2$ induces a weak equivalence of the categories $M_1\searrow pr_2 \we M_2\searrow pr_2$. Now, we apply Quillen's Theorem $B$ to show that the homotopy fiber of $pr_2$ over a cofibrant $M$ is weakly homotopy equivalent to the nerve of $M\searrow pr_2$. The proposition now follows from the diagram~\ref{pathspace}. \end{proof}

Under the same hypotheses as the previous proposition, we can consider what happens to derived mapping spaces under the adjunction $(E,U)$.

\begin{theorem}\label{homotopymapping}Suppose that $(E,U)$ is as above, and that the right adjoint $U,$ preserves all weak equivalences. Then there is a natural weak homotopy equivalence $\Map^h(M, U(X_f))\we\Map^{h}(E(M_c),X).$ \end{theorem}

\begin{proof}Assume that $M$ is cofibrant in ${}_{\cP}\M^*_{\Q}$ and that $X$ is fibrant in $\mathcal{B}$. We can construct a category $\mathcal{Z}$ which has objects $[M\bwe E(M_1)\ra X_1\bwe X]$ and morphisms are pairs $(E(\phi),f)$ where $\phi:E(M_1)\lra E(M_2)$ and $f:X_1\lra X_2$ are morphisms which make the obvious diagram commute. We will show that the nerve of $Z$ is weakly homotopy equivalent to $\Map^{h}(EM,X)$. \

Let $({}_{\cP}\M_{\Q})_c$ be the subcategory of cofibrant pointed $\cP$-$\Q$-bimodules; let $M\bwe M_1$ be a weak equivalence between cofibrant pointed bimodules. A functor $F:Z\lra w({}_{\cP}\M_{\Q})_{c}\searrow M$ is given by $$[M\bwe E(M_1)\ra X_1\bwe X]\lra [M\bwe M_1].$$ \

Fix the object $Y:=[M\bwe M_2]$. Then consider the over category $Y\searrow F$. Objects of $Y\searrow F$ can be re-written as $[M\bwe M_2\ra X_1\bwe X]$. The nerve of $Y\searrow F$ is weakly homotopy equivalent to $\Map^{h}(E(M_2),X)$. By the previous proposition, $\Map^{h}(E(M_2),X)$ is weakly homotopy equivalent to the homotopy fiber of the functor $F$. Since $w({}_{\cP}\M^*_{\Q})_c\searrow M$ has a terminal object, the moduli space $w({}_{\cP}\M^*_{\Q})_c\searrow M$ is contractible, and thus $\Map^{h}(E(M_2),X)$ is weakly homotopy equivalent to the nerve of $Z$. \

In a similar manner, we can consider a category $\mathcal{Z}'$ with objects that look like $$[M\bwe M_1\ra U(X_1)\bwe X].$$ A dual argument to above shows that $\mathcal{Z}'$ is weakly homotopy equivalent to $\Map^{h}(M, U(X_2))$. It is then clear that there is a weak homotopy equivalence $\mathcal{Z}\lra\mathcal{Z}'$. \end{proof} 

\begin{proof}[Proof of the Main Theorem]The subcategory of objects of $\cP+\Q\searrow\multi(\C)$ which satisfy the property that the natural map $\Q\ra\cP+\Q\ra\R$ is a weak equivalence is a category of zig-zags $[\cP\lra \R\bwe \Q]$. What's more, if we consider the restrictions of the functor pair $(E,U)$ to right quasi-free objects, then the functors fix objects. The theorem now follows from proposition~\ref{hofiber} and theorem~\ref{homotopymapping}.\end{proof}

\section{The Cosimplicial Model}As we mentioned in Section~\ref{MappingSpaces}, given that we assume $\Q$ is fibrant, we can also take a cosimplicial resolution of $\cP$ to obtain a model for $\Map^{h}(\cP,\Q)$. The fact that we have a fibrant replacement functor that fixes objects is key to what follows. The point of this section is to explicitly describe the components for this model of $\Map^{h}(\cP,\Q),$ from which we can then describe internal hom-objects.

\subsection{Cosimplicial Operads}The category $\multi(\C)$ is \emph{fibered} over varying sets of objects (\cite[1.6]{BM07}). A \emph{cosimplicial} multicategory is a cosimplicial object in this fibered category. More explicitly, a multicategory $\cP^{\bullet}$ is given by a cosimplicial \emph{set} of objects $C^{\bullet}$, such that for each $n$ we have an operad $\cP^n$, with object set $C^{n}$. The maps $\cP^n\leftrightarrows\cP^m$ are induced by arrows $C^n\leftrightarrows C^m$ that come from maps $[n]\lra[m]$ in $\Delta$.\

The \emph{geometric realization} of a cosimplicial operad $\cP^{\bullet}$ over $C^\bullet$, is given by a functor $\sSets\lra\multi(\C)$ which sends $X$ to the multicategory $|X|_{P^\bullet}:=X\otimes_\Delta P^\bullet$ with objects $X\otimes_\Delta C^\bullet$. An algebra over $|X|_{\cP^\bullet}$ consists of an algebra $A_x$ over the $C^n$-colored operad $P^n$, where we allow $x$ to vary over the $n$-simplices of $X$. 

As one might expect, this cosimplicial object $P^{\bullet}$ is completely determined by its $2$-skeleton (see~\cite[6]{BM07}).  This means that given a simplicial set $X$, the inclusion $sk_2(X)\lra X$ of the $2$-skeleton of $X$ induces a weak equivalence\begin{equation*}|sk_2(X)|_{P^{\bullet}}\overset{\cong}{\lra}|X|_{P^{\bullet}}\end{equation*}of multicategories. 

\begin{example}Let $\mathbb{A}^0=\mathbb{A}$ be the operad whose algebras are associative unitary monoids. Let $\mathbb{A}^1=BiMod$ be the operad with $3$-objects whose algebras are triples $(A_0,M,A_1)$, where $A_0,A_1$ are $\mathbb{A}^0$-algebras, and $M$ is an $A_0$-$A_1$-bimodule. The operads $\mathbb{A}^0$ and $\mathbb{A}^1$ form part of a cosimplicial operad $\mathbb{A}^{\bullet}$ over the cosimplicial set given by$$C^n=\{a_i\,|\,0\leq i\leq n\}\cup\{b_{ij}\,|\,0\leq i<j\leq n\},$$the objects for $n+1$ monoids $A_i$ and $\frac{n+1}{2}$ bimodules $M_{ij}$. Applying the Boardman-Vogt resolution, one gets a cosimplicial operad $W(\mathbb{A}^{\bullet})$. The $W(\mathbb{A}^0)$-algebras are $A_\infty$-algebras, the $W(\mathbb{A}^1)$-algebras are $\infty$-bimodules over such $A_{\infty}$-algebras, and so on.\end{example}In a similar manner, we can select any $\cP$ and create a cosimplicial $\cP^{\bullet}$ which parameterizes strings of morphisms between $\cP$-algebras described as follows. Let $\cP$ be $\cP^0$, and $\cP^n$ be the multicategory whose algebras are $n$-simplicies\begin{equation*}A_0\rightarrow A_1\rightarrow\cdots\rightarrow A_n\end{equation*}of $\cP$-algebra homomorphisms. \

In the case where $\cP$ has one object, the multicategory $\cP^{1}$ has been extensively studied by Markl~\cite{Mar04}. The operad $\cP^1$ has objects $\{0,1\}$ and algebras triples $(A_0,A_1,f)$ where $A_0$ and $A_1$ are $\cP$-algebras, and $f:A_0 \rightarrow A_1$ is a map of $\cP$-algebras. The operations can be given explicitly by\begin{equation}\cP^1(x_1,\dots,x_n;x)=\begin{cases}\cP(n)&\text{if }\max(x_1,\dots,x_n)\leq x;\\0&\text{otherwise}.\end{cases}\end{equation}Notice that a $\cP^1$-algebra consists of exactly two objects $A_0$ and $A_1.$ the structure of a $\cP$-algebra, and the unit morphism $1:I\rightarrow\cP(1)$ corresponds to a map $\alpha:I \rightarrow\cP^1(0;1),$ which corresponds to a homomorphism of $\cP$-algebras, $f:A_0\rightarrow A_1$. One can write out an explicit description of $P^n$ in a similar way.\

The multicategory $\cP^n$ is defined to be the pushout\begin{equation}\label{n-morphisms}\cP^n=\cP^1 \sqcup_{\cP^0}\cP^1\sqcup_{\cP^0}\cP^1...\cP^1\sqcup_{\cP^0}\cP^1,\end{equation} over $\theta^i:\cP^1\longrightarrow\cP^n$ where $\theta^i:[1]\lra[n]$ ranges over the inclusions $\{0,1\}$ to $\{i,i+1\}$, $0\le i\le n-1$. There are extension and restriction functors which are induced by the face and degeneracy maps(See, ~\cite[6]{BM07}). For a cosimplicial operad $\cP^{\bullet}$, the categories of algebras $\Alg_{\cP^n}(\M),\,n\geq 0$, together form a (very large) \emph{simplicial} category, which we denote by $\Alg_{\cP^{\bullet}}(\M)$. 

\begin{observation}The simplicial category $\Alg_{\cP^{\bullet}}(\M)$ is the nerve of the category $\Alg_{\cP}(\M)$.\end{observation} 

We can apply the $W$-construction to the cosimplicial operad $\cP^{\bullet}$ levelwise, and this provides a functorial cofibrant replacement $W(\cP^{\bullet}).$ In level $0$, $W(\cP^0)$ is just $W(\cP),$ i.e. a (functorial) cofibrant replacement for $\cP$. It was observed by Berger-Moerdijk~[6]\cite{BM07} that the category $\Alg_{W(\cP^{\bullet})}(\M)$ should be equivalent to the nerve of the category $\Alg_{\cP}(\M)$ up to weak equivalence. We use the results of the previous section to make this precise. \

Explicitly, we lift already know that ${}_{\cP}\M_{\Q}$ is weakly homotopy equivalent to $\Map^{h}(\cP,\Q)$ and, given a cosimplicial resolution of $\cP$ we should be able to construct a weak homotopy equivalence from ${}_{\cP}\M_{\Q}$ to a zig-zag of the form $[\cP\btrfib \cP^n \lra \Q]$ when $\Q$ is fibrant (see Section~\ref{MappingSpaces}). We will need the following theorem. 

\begin{theorem}\label{qe1}~\cite[Theorem 6.4]{BM07}Let $\M$ be a left proper, cofibrantly generated, monoidal model category over $\C$ and let $\cP$ be a $\Sigma$-cofibrant $\C$-enriched multicategory. If all of the cofibrant operads $W(\cP^n)$ all are admissible, then for $n\ge 2$, the map $\theta$ induces a Quillen equivalence\begin{equation}\label{eq4}\Alg_{W(\cP^1)}(\M)\times_{\Alg_{W(\cP^0)}(\M)}\times\cdots\times_{\Alg_{W(\cP^0)}(\M)}\Alg_{W(\cP^1)}(\M)\overset{\sim}{\longrightarrow}\Alg_{W(\cP^n)}(\M).\end{equation}\end{theorem}

Given a morphism $f:M\ra N$ between $\cP$-$\Q$-bimodules we know that we can construct a $\cP$-$\cP$-bimodule $\End(M,N)$ as in Theorem~\ref{invariance}. As in the previous section, we know that we can consider these bimodules as multicategories under $\cP+\cP$.  Our first proposition says that, under certain hypotheses, the endomorphism bimodule $\End(M,N)$ is a $\cP^1$-algebra in the category $\cP+\cP\searrow\multi(\C)$. \

\begin{prop}\label{E(f)1}Assume that $f:M\rightarrow N$ is a weak equivalence between cofibrant and fibrant pointed $\cP$-$\Q$-bimodules. Then there exists a $\cP^{1}$-algebra structure on $\End(M,N)$ in the homotopy category of $\cP^0+\cP^0\searrow\multi(\C)$.\end{prop} 

\begin{proof}By our assumption that $M$ and $N$ are cofibrant and fibrant right quasi-free $\cP$-$\Q$-bimodules we know that $\mathbb{L}E(M)=E(M)$ and $\mathbb{L}E(N)=N$. Moreover, we know that $E$ preserves weak equivalences between cofibrant objects, so that $E(f)$ is a weak equivalence. 

Now, by Theorem~\ref{invariance} we know that exists a pullback square (of right $\Q$-modules): 
\begin{equation}
\xymatrix{
\End(f)\ar[r] \ar[d] &  \End(M) \ar[d] \\
\End(N)\ar[r] & \End(M,N).
}\end{equation}The object $\End(f)$ is characterized by the fact that a morphism $\cP\ra\End(f)$ is equivalent to giving a $\cP$-algebra structure on $\End(N)$, a $\cP$-algebra structure on $\End(M)$in such a way that $f$ is a morphism of $\cP$-algebras. In other words, there exists a map $\cP^{1}\ra \End(f)$ and thus a $\cP$-$\cP$-bimodule via the $\cP$-$\cP$-bimodule structure on $\cP^{1}$.       

The collection $\End(M,N)$ also has a $\cP^0$-$\cP^0$-bimodule, via the maps $\cP\ra\End(M)\ra\End(M,N)$ and $\cP\ra\End(N)\ra\End(M,N)$. It is also the case that, by our assumptions that $M$ and $N$ are cofibrant and fibrant that $\End(M,N)$ is cofibrant, and so we can lift $\End(M,N)$ by $E$ to an object in $\cP^0+\cP^0\searrow\multi(\C)$. \ 

Now, we can consider either composite $$\cP^{1}\ra\End(f)\ra\End(M)\ra\End(M,N)$$or$$\cP^{1}\ra\End(f)\ra\End(N)\ra\End(M,N)$$ as a morphism between $\cP^1$ to $\End(M,N)$ in $(\cP^0+\cP^0)\searrow\multi(\C)$ and, in particular, we have the following diagram (as objects in $\cP+\cP\searrow \multi(\C)$).
$$
\xymatrix{
\cP+ \cP \ar[r]^{\alpha+\beta} \ar[d] & \End(M,N)  \\
\cP^1 &
}.$$ Now, since $\End(M,N)$ is fibrant and $\cP^1$ is cofibrant in this picture, we can repeat the proof of Theorem~\ref{invariance} to show that the $\cP$-$\cP$-bimodule structure on $\cP^{1}$ maps to a $\cP$-$\cP$-bimodule structure on $\End(M,N)$ in a homotopy coherent way. In particular, this tells us that $E(f)$ is represented by:
$$\xymatrix{
W(\cP)+ W(\cP) \ar[r]^{\alpha+\beta} \ar[d] & \End(A,B)  \\
W(\cP^1)\ar[ur]_{E(f)} &. }$$\end{proof} 

\begin{theorem}\label{cosimplicial} Assume that $\cP$ is $\Sigma_*$-cofibrant and $\Q$ is fibrant. Then $\Alg_{W\cP^\bullet}(\M_{\Q})$ is weakly homotopy equivalent to $\Map^{h}(\cP,\Q)$.\end{theorem}

\begin{proof}By Theorem~\ref{maintheorem} we know that ${}_{W(\cP)}\M_{\Q})$ is weakly homotopy equivalent to $\Map^{h}(\cP,\Q)$. We also know that there is a natural morphism of simplicial sets $[\underline{\Hom}]_{1}(W\cP,\Q)\lra {}_{W\cP^{\bullet}}\M^*_{\Q}$ which sends a string $A_1\ra ...\ra A_{n+1}$ to the corresponding $W(\cP^{n})$-algebra in ${}_{W\cP^{\bullet}}\M^*_{\Q}$. The remainder of the proof follows from Proposition~\ref{E(f)1} and Theorem~\cite[6.4]{BM07}.\end{proof}

The theorem basically provides us with a specific \emph{framing} (~\cite[Chapter 5]{Hovey}). As a consequence we know that when $K$ is a finite simplicial set, there exist maps in $\Ho(\sSets)$ $$\cP^{\mathbb{R}K}\longrightarrow \Map^h(K,\Alg_{W\cP^{\bullet}}(\Mod_{Q})) \qquad \cP^{\mathbb{R}K}\longrightarrow Map(K,\Map^{h}(\cP,\Q))$$ are in fact isomorphisms. Moreover, it implies that for any finite $K$ in $\Ho(\sSets)$ and any $\Sigma_*$-cofibrant multicategory $\cP$, we get an isomorphism$$[K,\Alg_{W\cP^{\bullet}}(\M_{\Q})] \longrightarrow [K,\Map^h(\cP,\Q)].$$

\begin{corollary}The monoidal category $Ho(\multi(\C))$ is closed. Furthermore, for any two simplicial multicategories $\cP$ and $\Q$ there is a natural isomorphism in $Ho(\multi(\C))$ $$\mathbb{R}\underline{\Hom}(\cP,\Q)\sim {}_{\cP}\M_{\Q}.$$ \end{corollary}

\section{Algebras over Multicategories} For a given (symmetric) multicategory $\cP$, a $\cP$-algebra $A$ is an object in the product category $\C^{\obj(\cP)}$ together with a left $\cP$-action, i.e. a collection of $\C$-morphisms\begin{equation*}\alpha_{x_1,...,x_n;x}:\cP(x_1,...,x_n;x)\otimes A(x_1)\otimes...\otimes A(x_n)\rightarrow A(x),\end{equation*}satisfying axioms for associativity, units and equivariance. A $\cP$-algebra homomorphism $f:A\rightarrow B$ is a family of $\C$-morphisms \begin{equation*}\{f:A(x_i)\rightarrow B(x_1)\}_{x_i\in\cP}\end{equation*}which fit into the following commutative diagram: 
$$\begin{CD}\\
\cP(x_1,...,x_n;x)\otimes A(x_1)\otimes...\otimes A(x_n) @>>> A(x)\\
@Vid\otimes f_{x_1}\otimes...\otimes f_{x_n}VV      @Vf_{x}VV\\
\cP(x_{1},...,x_n;x)\otimes B(x_{1})\otimes...\otimes B(x_n) @>>> B(x)\\
\end{CD}.$$ We denote the resulting category by $\Alg_{\cP}(\C)$.\\

Equivalently, a \emph{$\cP$-algebra} structure on an object $A\in\C^{\obj(\cP)}$, is a multifunctor $\cP\rightarrow\End(A)$ which fixes objects. The classifying object, $\End(A),$ is defined by \begin{equation*}\End(A)(x_1,...,x_n;x):=\Hom_{\C}(A(x_1)\otimes...\otimes A(x_{n}),A(x))\end{equation*} with composition (respectively $\Sigma_n$-actions) induced by substitution (respectively permutation) on the source factors. This object is called the \emph{endomorphism multicategory} of $A\in\C^{\obj(\cP)}$.

\begin{example} If $\cP$ is an ordinary category, then the category of $\cP$-algebras is the ordinary functor category $[\cP,\Sets]$. If $\cP$ is a \emph{strict monoidal category} then an algebra of the underlying multicategory is a lax monoidal functor from $\cP$ to $(\Sets, \times, 1)$. \end{example}

\begin{example} For each multicategory $\cP$ there exists an algebra $A$ defined by taking $A(x)$ to be the set $\cP(-;x)$ of arrows in $\cP$ from the empty sequence into $x$.  When $\cP$ is the multicategory of modules over some commutative ring $R$, this $\cP$-algebra is just the forgetful functor from $R$-modules to $\Sets$.\end{example}

\begin{example}There exists a multicategory $Op_{\C}$, whose category of algebras is the category of operads in $\C$. The set of objects in this case is the natural numbers $\mathbb{N}$.\

The elements of $Op(n_1,\dots,n_k;n)$ are equivalence classes of triples $(T,\sigma,\tau)$ where $T$ is a planar rooted tree with $n$input edges and $k$ vertices, $\sigma$ is a bijection $\{1,\dots,k\} \rightarrow V(T)$ (i.e. the set of vertices of $T$) with the property that the vertex $\sigma(i)$ has valence $n_i$ (i.e. $n_i$ input edges), and $\tau$ is a bijection $\{1,\dots,n\}\rightarrow in(T)$, the set of input edges of $T$.  Two such triples $(T,\sigma,\tau),(T',\sigma',\tau')$ represent the same element of $Op(n_1,\dots,n_k;n)$ if there is a (planar) isomorphism $\varphi:T\rightarrow T'$ with $\varphi \circ \tau=\tau'$ and $\varphi \circ \sigma=\sigma'$.\

Any $\alpha\in\Sigma_k$ induces a map $\alpha^{*}:Op(n_1,\dots,n_k;n)\longrightarrow Op(n_{\alpha(1)},\dots,n_{\alpha(k)};n)$ sending (the equivalence class of) $(T, \sigma, \tau)$ to $(T, \sigma\alpha, \tau)$.  The identity element $1_n \in Op(n;n)$ is represented by the tree $t_n$ (the corolla with $n$ leaves) whose inputs are numbered $1,\dots,n$ from left to right with respect to the planar structure.\

The composition product is defined as follows: given $(T,\sigma,\tau)$ as above, and $k$ other such $(T_1,\sigma_1,\tau_1),\dots,(T_k,\sigma_k,\tau_k)$, with $n_1,\dots,n_k$ inputs and $p_1,\dots,p_k$ vertices respectively, one obtains a new planar rooted tree $T'$ by replacing the vertex $\sigma(i)$ in $T$ by the tree $T_i$, identifying the $n_i$ input edges of $\sigma(i)$ in $T$ with the $n_i$ input edges of $T_i$ via the bijection $\tau_i$ (the $l$-th input edge of $\sigma(i)$ in the planar order is matched with the input edge $\tau_i(l)$ of $T_i$).\

The vertices of the new tree $T'$ are numbered in the following order: first the vertices of $T_{\sigma(1)}$ in the order given by $\sigma_1$, then the vertices in $T_{\sigma(2)}$ in the order given by $\sigma_2$, etc.  In other words, the map $\{1,\dots, p_1+\dots +p_k\}\rightarrow V(T')$ is given by $(\sigma_1\times\cdots\times\sigma_k)\circ\sigma(p_1,\dots,p_k)$ where $\sigma(p_1,\dots,p_k)$ permutes the blocks of size $p_i$. The new tree $T'$ still has $n$ input edges, which are ordered as given by $\tau$ and the identifications given by the $\tau_i$.  Notice that $Op(n_1;n)=\Sigma_n$ if $n_1=n$, and $Op(n_i,n)=\phi$ otherwise.  More precisely, $Op(n,n)$ consists of pairs $(t_n,\tau)$ where $t_n$ is the tree above and $\tau$ is a numbering of its inputs.  The composition product of $Op$ in particular gives a map $Op(n,n)\times Op(n,n)\longrightarrow Op(n,n)$ which sends $((t_n,\tau),(t_n,\rho))$ to $(t_n,\rho\tau)$, so that $Op(n;n)$ is identified with the \emph{opposite group} of $\Sigma_n$.\

The $Op$-algebras are exactly the operads in sets. Applying the strong symmetric monoidal functor $\Sets\rightarrow\C$ gives $Op_{\C}$ whose algebras are exactly the operads in $\C$. \end{example}

Let $\M$ be a $\C$-model category and $\cP$ a multicategory enriched in $\C$. Then there is an adjoint pair $$\xymatrix{F_{\cP} : \M^{\obj(\cP)} \ar@<3pt>[r] & \ar@<3pt>[l] \Alg_{\cP}(\M) : U_{\cP},}$$where $F_{\cP}$ is the free $\cP$-algebra functor defined by \begin{equation}\label{freealgebra}F_{\cP}(A)(x)=\coprod_{n\ge 0}\left(\coprod_{x_1,\dots, x_n\in \obj(\cP)}\cP(x_1,\ldots, x_n;x)\otimes_{\Sigma_n} A(x_1)\otimes\cdots\otimes A(x_n) \right)\end{equation} for every $A=(A(x))_{x\in\obj(\cP)}$ in $\M^{\obj(\cP)}$, and $U_{\cP}$ is the forgetful functor. If a simplicial multicategory $\cP$ is $\Sigma$-cofibrant, i.e. for each $x_1,...,x_n;x$ $\cP(x_1,...,x_n;x)$ is a cofibrant object in $\M$ then the model structure on $\M^{\obj(\cP)}$ is transferred to $\Alg_{\cP}(\M)$ along the free-forgetful adjunction (see~\cite{BM07}). 

\subsection{Endomorphism Modules for Algebras} Let us denote by $\M^{S}$ the product category of copies of $\M$ indexed by the set $\obj(\cP)=S$. For each $A=\{A(x)\}_{x\in S}$ we define a simplicial multicategory $\End(A)\in\multi(\C)_{S}$\begin{equation}\label{endomrphism operad}\End(A)(x_1,\ldots, x_n; x)=\Hom\nolimits_{\mathcal{C}}(A(x_1)\otimes\cdots\otimes A(x_n), A(x)),\end{equation} which forms a classifying object for $\cP$-algebra structures on $A$.\footnote{ We take $A(x_1)\otimes\cdots\otimes A(x_n)$ is taken to be the unit $\unit$ if $n=0$.}The monoidal product is ordinary composition in $\M$ and the $\Sigma_n$\nobreakdash-actions are given by permuting the source factors. The object $\End(A)$ is called the \emph{endomorphism multicategory} or \emph{endomorphism $S$-colored operad} of $A\in\C^{\obj(\cP)}$.\

We can define a similar object which provides a classifying object for the $\cP$-algebra homomorphisms $f\colon A\longrightarrow B$ in $\C^{S}$, i.e., an $S$\nobreakdash-indexed family of maps $\{f_x\colon A(x)\longrightarrow B(x)\}_{x\in S}$ in $\C$, there is an endomorphism object $\End(\mathbf{f})$, defined as the pullback of the following diagram of collections:
\begin{equation}
\xymatrix{
\End(f)\ar[r] \ar[d] &  \End(A) \ar[d] \\
\End(B)\ar[r] & \End(A,B).
}
\label{endomorphism object of a map}\end{equation} The collection $\End(A,B)$ is called an \emph{endomorphism module between $\cP$-algebras $A$ and $B$} and is defined\begin{equation}\End(A,B)(x_1,\ldots,x_n;x)=\Mor\nolimits_{\mathcal{C}}(A(x_1)\otimes\cdots\otimes A(x_n), B(x)).\end{equation}There are natural maps $\End(A)\longrightarrow \End(A,B)$ and $\End(B)\longrightarrow\End(A,B)$ which come from composing with $f$ on either side.\footnote{Set theoretically, $\End_f(n)=\{(\phi,\psi)\in\End_A(n)\times\End_B(n)\,|\,f\phi=\psi f^{\otimes n}\}$.} The collection $\End(f)$ inherits a multicategory structure from $\End(A)$ and $\End(B)$ (cf. \cite[Theorem~3.5]{BM03}). It also turns out that $\End(A,B)$ forms a left $\End(B)$-module and a right $\End(A)$ module. Moreover, these actions are compatible, giving us our first example of an operadic bimodule. We will discuss the endomorphism modules more in the next section.\footnote{If $P$ is a non\nobreakdash-symmetric multicategory, then endomorphism objects are defined in the same way, by forgetting the symmetric group action on $\End(A)$.}\

We choose the definition of $\End(f)$ so that it will provide a classifying object for $\cP$-algebra homomorphisms $f\colon A\longrightarrow B.$ More specifically, a multifunctor $\cP\longrightarrow\End(f)$ is equivalent to providing a $\cP$\nobreakdash-algebra structure on $A$ and a $\cP$\nobreakdash-algebra structure on $B$ in such a way that $f$ is $\cP$\nobreakdash-algebra map between them.\


Versions of the following theorem appear in \cite{Rezk},\cite{BM03},\cite{BM07},and \cite{BV}.  

\begin{theorem}\label{invariance}
Let $f:A\rightarrow B$ be a map between objects in the diagram category $\C^{S},$ and assume that $\C$ be a symmetric monoidal model category which satisfies all of the additional conditions necessary for $\multi(\C)_{S}$ to support a model category structure. Further, suppose that $\cP$ is a $\Sigma$-cofibrant object in $\multi(\C)_{S}$, i.e. that the underlying collection of $\cP$ is a cofibrant object in $\coll(\C)_{S}$.
\begin{enumerate}
  \item  Assume that $B$ is fibrant as an object in $\C^{S}$, and that $f^{\otimes n}$ is a trivial cofibration for each $n\geq 1$, then any $\cP$-algebra structure on $A$ extends along $f$ to a $\cP$-algebra structure on $B$.
  \item If $A$ is cofibrant as an object in $\C^{S}$, and $f$ is a trivial fibration, then any $\cP$-algebra structure on $B$ can be lifted along $f$ to a $\cP$-algebra structure on $A$.
  \item If both $A$ and $B$ are bifibrant objects in $\C^{S}$, and $f$ is a weak equivalence, then any $\cP$-algebra structure on $A$ (respectively $B$) induces a $\cP$-algebra structure on $B$ (respectively $A$) in such a way that $f$ preserves the $\cP$-algebra structures up to homotopy.\end{enumerate}
\end{theorem}

\begin{proof} We define the collections $\End(A,B)$ and $\End(f)$ as we did above. The key idea in the proof is that a morphism $f$ is compatible with the $\cP$-algebra structure maps $\cP\to\End_A$ and $\cP\to\End_B$ if and only if these are induced by an operad map $\cP\to\End(f)$. 

We are first assuming that $f$ has a fibrant target. The model category structure on $\C^{S}$ has weak equivalences and fibrations defined objectwise. Further, $\C^{S}$ has a symmetric monoidal tensor product induced by the symmetric monoidal tensor product from the category $\C$, and this structure is compatible with the model category structure, i.e. $\C^{S}$ supports the structure of a monoidal model category over $\C$. 

Now, given our assumption that $B$ is fibrant, and that $\C^{S}$ is a $\C$-model category, we can apply the pushout-product axiom, to show that the horizontal maps of the diagram 
\begin{equation}
\xymatrix{
\End(f)\ar[r] \ar[d] &  \End(A) \ar[d] \\
\End(B)\ar[r] & \End(A,B),
}
\end{equation} are trivial fibrations. The additional assumption that $\cP$ is a cofibrant operad, implies that the $\cP$-algebra structure map $\cP\rightarrow\End_A$ has a lift $\cP\rightarrow\End_f\rightarrow\End_B$ giving the required $P$-algebra structure on $B$.\

The hypothesis of $(2)$ are dual, implying that $\End_f\rightarrow\End_B$ is a trivial fibration, and that the the $P$-algebra structure map $\cP\rightarrow\End_B$ lifts to \begin{equation*}\cP\rightarrow\End_f\rightarrow\End_A. \end{equation*}\

Now assume that $f$ is a weak equivalence between cofibrant-fibrant objects and we assume that $A$ is a $\cP$-algebra, i.e. there exists a morphism $\cP\rightarrow\End_{A}$. We can factor $f$ into a trivial cofibration $f_1:A\rightarrow Z$ followed by a trivial fibration $f_2:Z\rightarrow B$. Since $f_2$ is a trivial fibration with a cofibrant target, we may assume that $f_2$ admits a trivial cofibration as section. Now consider the following pullback diagram:
\begin{equation}\xymatrix{
\End_{f_2}\ar[r]^{\phi} \ar[d] &\End_{Z}\ar[d]^{(f_2)_*}\\
\End_{B} \ar[r]^{(f_2)^*}&\End_{Z,B}.
}
\end{equation} Since $B$ is fibrant, we can again apply the pushout product axiom to conclude that each of the collections in the diagram is locally fibrant (equivalently, each of the collections is fibrant in the model structure on $\coll(\C)_{S}$). The vertical maps are trivial fibrations, and the horizontal maps are weak equivalences. As we assumed that $\cP$ is $\Sigma$-cofibrant, we have that the upper horizontal map $\phi$ induces a bijection$$[\cP,\phi]:[\cP,\End_{f_2}]\cong[\cP,\End_Z].$$

Now, the map $f_1$ is a trivial cofibration with a cofibrant source, and so satisfies the conditions of $(2)$. Therefore we can extend the $\cP$-algebra structure map $\cP\to\End_A$ to a $\cP$-algebra structure map $\phi_1:\cP\rightarrow\End_Z$. Since $\End(Z)$ is fibrant, and $\cP$ is $\Sigma$-cofibrant, we can lift the map $\phi_1:\cP\rightarrow\End_Z$ to a map $\psi:\cP\to\End_{f_2}$ such that $\phi_1$ and the composite $\phi\psi$ are homotopic, and this map is \emph{unique} up to homotopy. The composite map \begin{equation*}\cP\rightarrow\End_{f_2}\rightarrow\End_B\end{equation*} gives $B$ the structure of a $\cP$-algebra.\

We can make the dual argument so show that a $\cP$-algebra structure on $B$ induces a $\cP$-algebra structure on $A$ in such a way that $f$ preserves the $\cP$-algebra structures up to homotopy.\end{proof}

\begin{corollary}\label{composition} Now consider $f:M\buildrel{i}\over\rightarrow L\buildrel{p}\over\rightarrow N$ where $i$ is a cofibration and $p$ is a fibration in $\M$. Then the induced map $g:\End_{i,p}\rightarrow\End_{f}$ is a fibration. Further, $g$ is a weak equivalence if either $i$ or $p$ is a weak equivalence.\end{corollary}
\begin{proof} Define $\End_{i,p}:=\End_{M}\times_{\End_{M,L}}\End_{L}\times_{\End_{L,N}}\End_{N}.$ Then, by the earlier claim we know that $h:\End_{L}\rightarrow\End_{M,L}\times_{\End_{L,N}}\End_{N}$ is a fibration, which is trivial if either $i$ or $p$ is a weak equivalence. Then we notice that $g$ is defined as pullback over $h$:
\begin{equation}
\xymatrix{
g\ar[r]^{g} \ar[d]^{f} &  \End_{M} \ar[d]^{\End_{M,L}} \\
\End_{N}\ar[r]^{\End_{L,N}} & h.} \end{equation}\end{proof}

\section{Appendix: The Homotopy Theory of Operadic Bimodules}A \emph{left $\cP$-module} is an object $M$ in $\coll(\C)$ together with a left $\cP$-action $\cP\circ M\longrightarrow M$. A right \emph{$\Q$-module} is an object $N$ in $\coll(\C)$ together with a right action $N\circ\Q\longrightarrow N$.\ 

\begin{definition}
For any two multicategories, $\cP$ and $\Q$, a \emph{$\cP$-$\Q$-bimodule} is consists of an object $M$ in $\coll(\C)_{\obj(\cP)\times\obj(\Q)}$ which has a left $\cP$ action and a compatible right $\Q$ action:
\begin{itemize} 
	\item for each $a_1, \ldots, a_n \in \Q$ and each $b\in \cP$, an $\M$-object $M(a_1, \ldots, a_n; b)$ 
	\item for each $a_i^j \in\Q$ and $b_i, b \in\cP$, a left $\cP$-action ($\M$-morphism)
\begin{eqnarray*}
\begin{array}[b]{r}
\cP(b_1, \ldots, b_n; b) \otimes
M(a_1^1, \ldots, a_1^{k_1}; b_1) \otimes\cdots \\
\otimes 
M(a_n^1, \ldots, a_n^{k_n}; b_n)	
\end{array}
&
\rightarrow	&
M(a_1^1, \ldots, a_n^{k_n}; b),	\\
(\phi, \xi_1, \ldots, \xi_n)	&
\mapsto	&
\phi \cdot (\xi_1, \ldots, \xi_n)
\end{eqnarray*}
\item for each $a_i^j, a_i \in \Q$ and each $b\in\cP$, a right $\Q$-action ($\M$-morphism)
\begin{eqnarray*}
\begin{array}[b]{r}
M(a_1, \ldots, a_n; b) \otimes
\Q(a_1^1, \ldots, a_1^{k_1}; a_1) \otimes\cdots \\
\otimes \Q(a_n^1, \ldots, a_n^{k_n}; a_n)	
\end{array}
&
\rightarrow	&
M(a_1^1, \ldots, a_n^{k_n}; b),	\\
(\xi, \theta_1, \ldots, \theta_n)	&
\mapsto	&
\xi\cdot (\theta_1, \ldots, \theta_n),
\end{eqnarray*}
\end{itemize}which satisfy the evident axioms for compatibility with the composition products and identities of both $\Q$ and $\cP$, in addition to:
\[
(\phi \cdot (\xi_1, \ldots, \xi_n))
\cdot 
(\theta_1^1, \ldots, \theta_n^{k_n})
=
\phi \cdot
(\xi_1 \cdot (\theta_1^1, \ldots, \theta_1^{k_1}),
\ldots,
\xi_n \cdot (\theta_n^1, \ldots, \theta_n^{k_n}))
\]whenever these expressions make sense.\end{definition}

The morphisms between $\cP$-$\Q$-bimodules are maps of collections which are compatible with both the left $\cP$-action and the right $\Q$-action. We denoted the resulting category by ${}_{\cP}\M_{\Q}$. In the special case where $\cP$ and $\Q$ have only unary arrows, we recover the usual definition of bimodule between enriched categories (sometimes called a \emph{pro-functor}). If $\cP$ and $\Q$ are both operads, then we recover Rezk's definition of a $(\cP,\Q)$-biobject~\cite{Rezk}. 

\begin{example}Every multicategory is itself a $\cP$--$\cP$-bimodule.\end{example} 

\begin{example}\label{unit}Let $\cP = I$ be the trivial operad, so that $\Alg_{I}(\C) = \C.$ An $I$-$I$–-bimodule $M$ is just a symmetric sequence in $\C$.\end{example} 

The category of right $\Q$-modules is a closed, symmetric monoidal category over $\C$. This fact requires some checking (see~\cite{FresseBook}), but if we accept that the category of collections in $\C$ is a closed, symmetric monoidal category over $\C$ with respect to the pointwise tensor product, then it remains to check that this structure lifts to the category of objects with right $\Q$-action. Since the circle product commutes with colimits on the \emph{left}, the tensor product of two right $\Q$-modules has a natural right $\Q$-module structure, where\begin{equation*} M\circ\Q\otimes N\circ\Q = (M\otimes N)\circ\Q.\end{equation*}

\begin{lemma}Let $\Q$ be a multicategory enriched in $\C$. The category of right $\Q$-modules is a cocomplete, closed, symmetric monoidal category over $\C$.\end{lemma}

\begin{theorem}\label{model structure on right modules}Let $\Q$ be a multicategory which is locally cofibrant, i.e. for every $n\ge0$ and every sequence $x_1,...,x_n;x$ the $\C$-object $\Q(x_1,...,x_n;x)$ is a cofibrant object in $\C$. Then the category of right $\Q$-modules admits a cofibrantly generated monoidal model category over $\C$. Moreover, if $\C$ is right (respectively, left) proper then so is the model category structure on $\M_{\Q}$.\end{theorem}

\begin{remark}Throughout this paper we have been using the model structure from~\cite{BM03} on simplicial operads (or a generalization there of), but in the thesis~\cite{Rezk}, Rezk has a different, though Quillen equivalent, model structure on the category of simplicial operads. The reader could choose to use Rezk's model structure (properly generalized), and then it would not be necessary to assume that your multicategory is locally cofibrant.\end{remark} 

The theorem follows from somewhat standard arguments that the forgetful functor from right $\Q$-modules to collections\begin{equation*} -\circ \Q:\coll(\C)\rightleftarrows\M_{\Q}:U\end{equation*}preserves and detects weak equivalences and fibrations. Since there are at least two proofs known to the author of this theorem in the one-object case (See,~\cite{Rezk,FresseBook}) which easily generalize to the many objects case\footnote{Note that here ``many objects'' is still referring to fixed object sets since the category of right $\Q$-modules is really collections with $|\obj(\Q)|$ objects which are equipped with a right $\Q$-action.}, we will not include a complete proof here. The important thing for us is that we can describe the generating (acyclic) cofibrations as tensor products of generating (acyclic) cofibrations of the base category $\C$. The distribution relation between the composition product and the external tensor product gives identifications $(i\otimes K)\circ \R = i\otimes(K\circ \R).$Now, if the tensor product $-\otimes D: \C\rightarrow\C$ maps acyclic cofibrations to weak-equivalences for all $D\in\C$, then it follows immediately that the model structure lifts from $\coll(\C)$ to $\M_{\Q}$. Otherwise, we can use our assumption that $\Q$ is locally cofibrant to show that the objects $K\circ \Q$ are locally cofibrant. It follows that the tensor products $-\otimes(K\circ\Q))$ preserve acyclic cofibrations.

\begin{lemma}The generating (acyclic) cofibrations of the model category structure on $\M_{\Q}$ are given by \begin{equation*}i\circ\Q: K\circ\Q\longrightarrow L\circ \Q, \end{equation*} where $i:K\rightarrow L$ is a generating (acyclic) cofibration of $\coll(\C)$.\end{lemma}

\begin{prop}The model category of right $\Q$-modules is a monoidal model category over $\C$.\end{prop}

\begin{proof}The claim follows from the fact that the category of $\C$-collections forms a symmetric monoidal model category over $\C$ (generalize~\cite[14.1]{FresseBook} to fixed objects case), and the description of the generating (acyclic) cofibrations for right $\Q$-modules.\end{proof}

\begin{prop}~\cite[14.1]{FresseBook}If $\C$ is a right (respectively, left) proper model category and $\Q$ is locally cofibrant, then $\M_{\Q}$ is right (respectively, left) proper as well.\end{prop}

We can now prove the following theorem.\

\begin{theorem}[Model Structure]Let $\cP$ and $\Q$ be two simplicial multicategories. If the category of right $\Q$-modules admits a cofibrantly generated model category structure then the category $[\underline{\Hom}]_{1}(\cP,\M_{\Q})$ admits a cofibrantly generated model category structure where a homomorphism $f:A\rightarrow B$ is a weak equivalence (respectively, fibration) if the underlying map of right $\Q$-modules, $Uf:U(A)\rightarrow U(B)$, is a weak equivalence (respectively, fibration).\end{theorem}

Recall that given any cocomplete category $\M$ and any class of maps $I$ in $\M$ then the subcategory of \emph{relative $I$-cell complexes} is the subcategory which can be constructed via transfinite compositions and pushouts of the maps in $I$. 

\begin{lemma}[Classifying Fibrations]\label{classifyingfibrations}A map in $[\underline{\Hom}]_{1}(\cP,\M_{\Q})$ is a fibration if, and only if, it has the right lifting property with respect to retracts of relative $F_{\cP}(\iota_{*}J)$-cell complexes.\end{lemma}

\begin{proof}Fibrations of generalized $\cP$-algebras were defined via the free-forgetful adjunction. In particular, $f:A\rightarrow B$ is a fibration of $\cP$-algebras if, and only if, $Uf:A\rightarrow B$ is a fibration of right $\Q$-modules. The map $Uf$ is a fibration if, and only if, $Uf$ has the right lifting property (RLP) with respect to a retract of something in $\iota_{*}J.$ The lemma then follows from adjunction.\end{proof}

\begin{lemma}[Classifying Trivial Fibrations]\label{classifying trivial fibrations}A map in $[\underline{\Hom}]_{1}(\cP,\M_{\Q})$ is a trivial fibration if and only if it has the right lifting property with respect to retracts of relative $F_{\cP}(\iota_{*}I)$-cell complexes.\end{lemma}

\begin{lemma}A relative $F_{\cP}(\iota_{*}J)$ complex is a weak equivalence of right $\Q$-modules.\end{lemma}

\begin{proof}The proof closely follows the proof of Lemma 6.2~\cite{SS98} and will depend on analyzing pushouts in $[\underline{\Hom}]_{1}(\cP,\M_{\Q})$. Let $A$ be an object in the category $[\underline{\Hom}]_{1}(\cP,\M_{\Q})$ and let $M\longrightarrow N$ be map of objects in $\M_{\Q}$ and recall that $\iota_{x}:\M_{\Q}\longrightarrow\M_{\Q}^{S}$ is the left adjoint to the ``evaluation at object $x$'' functor. We will want to study the pushout of the diagram
$$\begin{CD}
F_{\cP}(\iota_{x}M) @>>> F_{\cP}(\iota_{x}N)\\
@VVV \\
A
\end{CD}$$in $[\underline{\Hom}]_{1}(\cP,\M_{\Q})$. The pushout will be given as a colimit, taken in right $\Q$-modules, of a sequence $$A=X_0\rightarrow X_1\rightarrow...\rightarrow X_n\rightarrow... .$$

We will build up this filtration in layers. First, let's understand $A\coprod F_{\cP}(M)$ by building up the $k$-ary relations. For each $k\ge0$ and each sequence of objects $x_1,...,x_k$ in $\cP$ we will construct an object $G_{x_1,...,x_n}(A)$ as the coequalizer of the following diagram\begin{eqnarray*}\coprod_{n\ge0}\big(\coprod_{y_1,...,y_n}\cP(y_1,...,y_n,x_1,...,x_k;-)\otimes_{\Sigma_{n}} F_{\cP}(A(y_1))\otimes...\otimes F_{\cP}(A(y_n)\big)\\
\rightrightarrows \coprod_{n\ge0}\big(\coprod_{y_1,...,y_n}\cP(y_1,...,y_n,x_1,...,x_k;-)\otimes_{\Sigma_{n}} A(y_1)\otimes...\otimes A(y_n)\big)\longrightarrow G_{x_1,...,x_k}(A).\end{eqnarray*}The top map comes from the structure of the multicategory $\cP$ and the bottom map comes from the action of the triple $F_{\cP}$ on $A$. The underlying object of the coproduct $A\coprod F_{\cP}(M)$ is$$\coprod_{k}\big(\coprod_{x_1,...,x_n}G_{x_1,...,x_k}(A)\otimes_{\Sigma_{n}} M(y_1)\otimes...\otimes M(y_n).$$ 

Now, let $f:M\rightarrow N$ be a map in $\M_{\Q}$. We will construct a right $\Q$-module $C_{k,i}(f)$ where $k\ge 0$ and $0\le i\le k$ as follows. Let $C_{k,0}=M^{\otimes k}$ and for $0<i<k$ we define $C_{k,i}$ as a pushout 

$$\begin{CD}
\unit[\Sigma_{k}]\otimes_{\Sigma_{k-i}\times\Sigma_{i}} M^{\otimes k-i}\otimes C_{k,i-1} @>>> \unit[\Sigma_{k}]\otimes_{\Sigma_{k-i}\times\Sigma_{i}} M^{\otimes k-i}\otimes N^{\otimes i}\\
@VVV                 @VVV   \\
C_{k,i-1}   @>>>     C_{k,i}.
\end{CD}$$

We can combine these constructions to get a filtration on the pushout of 
$$\begin{CD}
F_{\cP}(\iota_{x}M) @>>> F_{\cP}(\iota_{x}N)\\
@VVV \\
A.
\end{CD}$$ Let $X_{0}=A$, fix a $k$ and let fix an object $x$ in $\cP$, i.e. $x_1=...=x_k=x.$ We will define $X_{k}$ as the pushout in $\M_{\Q}^{S}$ 
$$\begin{CD}
G_{x_1,...,x_k}(A)\otimes_{\Sigma_{k}}C_{k,k-1} @>>> G_{x_1,...,x_k}(A)\otimes_{\Sigma_{k}}(\iota_{x}N)^{\otimes i}\\
@VVV                 @VVV   \\
X_{k-1}   @>>>     X_k.
\end{CD}$$The left hand vertical map comes from the map $\iota_{x}M\rightarrow A$ in $\M_{\Q}^{S}$. The pushout of the diagram is then $X=colim X_{k}$.\

Now, we may assume that $f:M\rightarrow N$ is of the form $K\circ \Q\rightarrow L\circ \Q$ where $K\rightarrow L$ runs over generating acyclic cofibrations of $\coll(\C)$. Our goal is to show that the map $A\rightarrow A\coprod_{F_{\cP}(\iota_{x}M)}F_{\cP}(\iota_{x}N)$ is a weak equivalence of right $\Q$-modules. But, since it is clear that for each $k\ge0$ $X_{k-1}\rightarrow X_{k}$ is a weak equivalence, we are done. \end{proof}

\begin{proof}[Proof of Theorem]The simplicial category $[\underline{\Hom}]_{1}(\cP,\M_{\Q})$ is cocomplete by~\ref{bicomplete} and it is clear that the class of weak equivalences and fibrations are closed under retracts and that the class of weak equivalences satisfies the ``2-out-of-3'' property. Cofibrations are defined via a lifting property, and so it is easy to check that they are closed under retracts.\ 

Given an arbitrary map $f$ in $[\underline{\Hom}]_{1}(\cP,\M_{\Q})$ we can apply the small object argument to produce a factorization $f=p\circ i$ where $i$ is in $F_{\cP}(\iota_{*}I)$ and $p$ has the right lifting property with respect to $i$. The our lemma~\ref{classifying trivial fibrations} implies that $p$ is an acyclic fibration. In a similar manner, we factor $f=q\circ j$, where $j$ is in $F_{\cP}(\iota_{*}J)$ and $q$ has the right lifting with respect to $j$. The lemma~\ref{classifyingfibrations} implies that $q$ is a fibration of multicategories.\

Finally, we check that given the square 

$$\begin{CD}
A @>f>> C\\
@VViV        @VVpV\\
B @>g>> D
\end{CD}$$ with $i$ a cofibration and $p$ a fibration. If $p$ is also a weak equivalence, then we find a lift by the classification of cofibrations. If $i$ is a weak equivalence, then we factor $i=q\circ j:A\hookrightarrow\widetilde{B}\twoheadrightarrow B$ where $q$ is a fibration and $j$ is an acyclic cofibration. Since we have shown that every acyclic cofibration is a weak equivalence, we know that $j$ is a weak equivalence. The ``2-out-of-3'' property for weak equivalences now implies that $q$ is an acylic fibration.\

Since $j$ is an acyclic cofibration, and $p$ is a fibration, we know that $p$ has the RLP with respect to $j$.  In other words, we have a lift $h:\widetilde{B}\longrightarrow C$ so that $j\circ h =f$ and $q\circ g=p\circ h$. \

Now, since $i$ is an acyclic cofibration and $q$ is a trivial fibration, there exists a retract $s$ of $q$ with $s\circ i = j$. The composite $h\circ s$ provides the desired lift.\end{proof}

\begin{corollary}The category $[\underline{\Hom}]_{1}(\cP,\M_{\Q})$ is a simplicial model category.\end{corollary} 

\begin{proof}The tensor of an $F_{\cP}$-algebra $A$ and a simplicial set $K$ is given by a reflexive coequalizer$$F_{\cP}((F_{\cP}A)\otimes K)\rightrightarrows F_{\cP}(A\otimes K)\longrightarrow A\otimes X.$$ We may assume that $i$ is a map in either $F_{\cP}(\iota_*I)$ or $F_{\cP}(\iota_*J)$. Since the monad $F_{\cP}$, viewed as a functor, is left adjoint to the forgetful functor we can reduce this issue to proving that the product category $\M_{\Q}^{S}$ satisfies $SM7$, which then reduces to showing that $\M_{\Q}$ satisfies $SM7$, and we it is known that $\M_{\Q}$ is a simplicial model category.\end{proof} 

It is now easy to check that the category of pointed $\cP$-$\Q$-bimodules is simplicially Quillen equivalent to $[\underline{\Hom}]_{1}(\cP,\M_{\Q})$. From this point on we will refer to $[\underline{\Hom}]_{1}(\cP,\M_{\Q})$ as the category of pointed $\cP$-$\Q$-bimodules and denote this category by ${}_{\cP}\M^*_{\Q}$.

\end{document}